\documentclass[11pt,a4paper]{article}
\usepackage[T1]{fontenc}
\usepackage{amsfonts}
\usepackage{amsmath}
\usepackage{empheq}
\usepackage{amssymb}
\usepackage{enumitem}
\usepackage{ wasysym }
\setlist{leftmargin=0mm}
\usepackage{setspace}
\usepackage{skull}
\usepackage{nicematrix}
\usepackage{tikz,pgfplots}
\pgfplotsset{compat=1.18}
\usepackage[english]{babel}
\usepackage{hyperref}
\usepackage{comment}
\usepackage{amsthm}
\usepackage{changepage}   
\usepackage{dsfont}
\usepackage[
backend=biber,
style=alphabetic,
sorting=nyt,
maxbibnames=99]{biblatex}
\usepackage{graphicx}
\usepackage{caption}
\usepackage[pagewise]{lineno} 
\newtheorem{theorem}{Theorem}
\newtheorem{cor}[theorem]{Corollary}
\newtheorem{prop}[theorem]{Proposition}
\theoremstyle{definition}
\newtheorem{definition}{Definition}
\newtheorem{lemma}[theorem]{Lemma}
\newtheorem{remark}[theorem]{Remark}

\newtheorem{notation}{Notation}

\usepackage{lipsum}

\newcommand\phantomarrow[2]{
  \setbox0=\hbox{$\displaystyle #1\to$}%
  \hbox to \wd0{%
    $#2\mapstochar
     \cleaders\hbox{$\mkern-1mu\relbar\mkern-3mu$}\hfill
     \mkern-7mu\rightarrow$}%
  \,}
\makeatletter
\newcommand*{\rom}[1]{\expandafter\@slowromancap\romannumeral #1@}
\makeatother

\newcommand{\R}{\mathbb{R}}

\newcommand{\N}{\mathbb{N}}

\DeclareMathOperator{\im}{im}
\newcommand{\T}{\mathbb{T}}

\newcommand{\I}{\mathds{1}}

\newcommand{\tend}[2]{\underset{#1 \to #2}{\longrightarrow}}
\newcommand{\tendf}[2]{\underset{#1 \to #2}{\rightharpoonup}}

\DeclareMathOperator{\cE}{{\cal E}}
\DeclareMathOperator{\cM}{{\cal M}}

\DeclareMathOperator{\Div}{div}
\DeclareMathOperator{\cv}{\mathfrak{c}_v}
\DeclareMathOperator{\sign}{sign}
\DeclareMathOperator*{\esssup}{ess\,sup}
\DeclareMathOperator*{\essinf}{ess\,inf}
\usepackage{csquotes}
\usepackage{xcolor} 
\title{Global in time justification of a two-phase averaged system for heat-conducting ideal gases\thanks{This work was supported by the ANR under France 2030 bearing the reference ANR-23-EXMA-004 (Complexflows project)}}
\date{\today}
\author{Pierre Gonin-{}-Joubert
\thanks{Universit\'e Claude Bernard Lyon 1, ICJ UMR5208, CNRS, \'Ecole Centrale de Lyon, INSA Lyon, Université Jean Monnet, 69622 Villeurbanne, France \texttt{goninjoubert@math.univ-lyon1.fr}}
\thanks{LAMA UMR5127 CNRS, Université Savoie Mont Blanc, Le Bourget du lac, France}}

\begin{document}

\maketitle
\begin{abstract}
\noindent In this article, we mathematically justify (globally in time) a Baer-Nunziato type system from the non-isentropic compressible Navier-Sokes equations for heat conducting ideal gases posed over the torus and in one space dimension. The breakthrough in this paper is to define and prove the global existence of solutions in a framework intermediate between weak and strong solutions and then to derive the system through homogenization and Young measures characterization. Note that the main difficulty is to derive a priori uniform bounds on appropriate unknowns in the presence of piecewise constant coefficients (viscosity and adiabatic constants) exhibiting rapid oscillations between two positive values. 
\end{abstract}
\section{Introduction}
This article deals with the homogenization of the non-isentropic Navier-Stokes equations. In one space dimension, this system can be written as

\begin{empheq}[left=\empheqlbrace]{alignat=1}
\partial_t \rho + \partial_x (\rho u) &= 0,
\label{eq:cont}\\
\partial_t (\rho u) + \partial_x (\rho u^2) &=\partial_x(\mu\partial_x u - p),
\label{eq:momts}\\
\partial_t \left(\frac{\rho u^2}{2} + \rho e\right) + \partial_x \left(\frac{\rho u^3}{2} + \rho e u\right)  &= \partial_x(\mu(\partial_x u)u  - pu) +\partial_x(\kappa\partial_x\theta),
\label{eq:energ}
\end{empheq}

where $\rho>0$ is the density, $u$ the velocity field, $\theta>0$ the temperature and $e>0$ the internal energy. The first equation \eqref{eq:cont} (continuity equation) describes the conservation of mass. The second equation (momentum equation) can be interpreted as the second Newton law of motion ($d{\bf{p}}/dt = {\bf F}$). The left hand side is the momentum, and the right hand side corresponds to the derivative of the sum of two forces: the viscous force $\mu\partial_x u$ (with $\mu>0$ some viscosity constant) and the pressure force $-p$. In general, $p$ is a function of the density and the temperature. In our analysis, we will consider ideal gases, so
\begin{equation}
p = R\rho\theta
\end{equation}
where $R>0$ is the constant of ideal gases. Finally, the third equation \eqref{eq:energ} is the total energy equation. The left hand side is some derivative of the total energy 
\begin{equation}
E = \frac{u^2}{2} + e
\end{equation}
composed of a kinetic energy part and an internal energy part. In the case of an ideal gas, we have
\begin{equation}
e = \cv\theta
\end{equation}
where $\cv>0$ is a constant, the specific heat. The right hand side of \eqref{eq:energ} can be interpreted as an energy flux. Observe that we consider the molecular shaking: this is the conductivity term $\partial_x(\kappa\partial_x\theta)$
where $\kappa\geq 0$ is the heat-conductivity constant. Another interesting physical quantity related to the problem is the entropy $s$, defined in the case of ideal gases by
\begin{equation}\label{eq:entropy}
s = \cv\ln\theta - R\ln\rho.
\end{equation}
In particular, denoting 
\begin{equation}
\gamma = \frac{R}{\cv} + 1 > 1
\end{equation}
the adiabatic constant, we get
\begin{equation}
p = R\rho^\gamma \exp(s/\cv).
\end{equation}
In the barotropic isentropic case, $s$ is constant and we find $p\sim \rho^\gamma$. In the non-isentropic case, from \eqref{eq:cont}--\eqref{eq:energ} we obtain
\begin{equation}\label{eq:entropymeso}
\partial_t(\rho s) + \partial_x(\rho s u) = \frac{\mu(\partial_x u)^2}{\theta} + \frac{\kappa(\partial_x\theta)^2}{\theta^2} - \partial_x(\kappa\partial_x\ln\theta).
\end{equation}
The purpose of this article is not to study a single pure fluid but a mixture of fluids. To do this, we proceed by homogenization, following the method of \cite{BrBuLa1}. Consider a mixture of two fluids, and assume that at a certain scale, they are not intermixed and occupy the entire space. Let us assume that the fluid $+$ is present only in the open set $\Omega_+$, and fluid $-$ in the open set $\Omega_-$, such that
\begin{equation}
\T = \overline{\Omega_+} \cup \overline{\Omega_-},\quad \Omega_+\cap\Omega_- = \varnothing.
\end{equation}
\begin{figure}[h]
  \centering
  \includegraphics[scale=1]{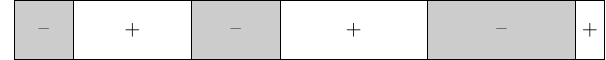}
  \caption{Occupation of the torus by + and -- fluids.}
  \label{fig:melange}
\end{figure}

At this scale, we refer to the system as a \textit{mesoscopic mixture}. We can then assume that there exists sufficiently regular functions $\rho_+, u_+, \theta_+$ (resp. $\rho_-, u_-, \theta_-$), the density, velocity, and temperature of the fluid $+$ (resp. the fluid $-$) such that
\begin{equation}
\rho = \rho_+\I_{\Omega_+} + \rho_-\I_{\Omega_-},\quad u = u_+\I_{\Omega_+} + u_-\I_{\Omega_-},\quad\theta = \theta_+\I_{\Omega_+} + \theta_-\I_{\Omega_-}\quad\text{a.e.}
\end{equation}
The two fluids may also have different physical constants $\mu_\pm,\cv_\pm,\gamma_\pm,R_\pm,\kappa_\pm$. The quantity $c=\I_{\Omega_+}$ is called the color function: it is the function that encodes which fluid is present at each point in space and time.
Let us consider two particles of fluid $x, y\in \Omega_+$. Then $x$ and $y$ satisfy
\begin{equation}
\frac{d}{dt}x(t) = u_+(t,x(t)),\quad\frac{d}{dt}y(t) = u_+(t,y(t)).
\end{equation}
Moreover, the mass of fluid in the segment $[x,y]$ is constant during the time, i.e.
\begin{equation}
\frac{d}{dt}\int_{x(t)}^{y(t)} \rho = 0.
\end{equation}
Moreover, on $[x,y]$, we get $\rho = \rho_+ =\rho c$ and $u = u_+ = uc$. Thus
\begin{equation}
\frac{d}{dt}\int_{x(t)}^{y(t)} \rho c = \rho u c(t,y(t)) - \rho u c(t,x(t)) + \int_{x(t)}^{y(t)}\partial_t(\rho c) = \int_{x(t)}^{y(t)} \partial_t(\rho c) + \partial_x(\rho u c).
\end{equation}
By the same reasoning on $\Omega_-$, we obtain
\begin{equation}\label{eq:cprems}
\partial_t (\rho c) + \partial_x(\rho c u) =0.
\end{equation}
From \eqref{eq:cont}, \eqref{eq:cprems} and supposing $\rho>0$, we get
\begin{equation}
\partial_tc + u\partial_x c = 0.
\end{equation}
A system describing the mesoscopic mixture is therefore the system with unknowns $(c,\rho,u,\theta)$, with initial conditions
\begin{equation}
\begin{split}
&c(0,\cdot)=c_0,\qquad \rho(0,\cdot) = \rho_{+,0}c_0 + \rho_{-,0}(1-c_0),\\& u(0,\cdot) = u_{+,0}c_0 + u_{-,0}(1-c_0),\qquad\theta(0,\cdot) = \theta_{+,0}c_0 + \theta_{-,0}(1-c_0),
\end{split}
\end{equation}
where $c_0,\rho_{\pm,0},u_{\pm,0},\theta_{\pm,0}\in L^2(\T)$, $c_0\in \{0,1\}$, and with equations 
\begin{empheq}[left=\empheqlbrace]{alignat=1}
\partial_t c + u\partial_x c &= 0,\label{eq:cmeso}\\
\partial_t \rho + \partial_x (\rho u) &= 0,
\label{eq:contmeso}\\
\partial_t (\rho u) + \partial_x (\rho u^2) &=\partial_x(\mu(c)\partial_x u - R(c)\rho\theta),
\label{eq:momtsmeso}\\
\partial_t \left(\frac{\rho u^2}{2} + \cv(c)\rho\theta\right) + \partial_x \left(\frac{\rho u^3}{2} + \cv(c)\rho \theta u\right)  &= \partial_x(\mu(c)(\partial_x u)u  - R(c)\rho\theta u) +\partial_x(\kappa(c)\partial_x\theta),
\label{eq:energmeso}
\end{empheq}
where for $f\in \{\mu,\cv,\gamma,R,\kappa\}$, $f(c)$ is defined as
\begin{equation}
f(c) = f_+ c + f_-(1-c).
\end{equation}
Homogenization then consists of assuming that the characteristic length of each slice in Figure \ref{fig:melange} tends towards $0$. This is a change of scale: starting from a \textit{mesoscopic} mixture, then zooming out, we obtain a \textit{macroscopic} mixture at infinity, for which both the + and -- fluids are present at each point in space, with a certain volume fraction $\alpha_\pm$. The dynamics of this macroscopic mixture will be described by equations obtained by passing to the limit in \eqref{eq:cmeso}--\eqref{eq:energmeso}. 

This limit transition was performed in the case $\kappa_\pm = 0$ in \cite{BrBuGJLa}, inspired by the proof in the barotropic case \cite{BrBuLa1, BrBuLa2}. The present article deals with the case $\kappa_\pm > 0$.

The principal difficulty encountered in the limit process is that of nonlinear terms. In order to characterise these terms, it is important to distinguish smooth quantities uniformly in $\varepsilon$, so as to apply strong compactness results such as the Rellich-Kondrachov theorem or the Aubin-Lions lemma. Note that the momentum equation \eqref{eq:momtsmeso} has a partially parabolic structure due to the viscosity term. Therefore, even if the velocity is initially chosen to be oscillatory, it is expected to be immediately regularised uniformly in $\varepsilon$. Equations \eqref{eq:cmeso} and  \eqref{eq:contmeso} and the regularity of the velocity indicate transport on $c$ and $\rho$. Since these quantities initially oscillate strongly in $\varepsilon$, these oscillations will be preserved over time and no better regularity than $L^\infty$ can be expected for them. Furthermore, from equations \eqref{eq:momtsmeso} and \eqref{eq:energmeso} we obtain
\begin{equation}\label{eq:thetameso}
\partial_t (\cv(c)\rho\theta) + \partial_x(\cv(c)\rho\theta u) = (\mu(c)\partial_x u - R(c)\rho\theta)\partial_x u + \partial_x(\kappa(c)\partial_x\theta).
\end{equation}
Thus, based on similar considerations, the behaviour of the temperature over time fundamentally depends on the choice of $\kappa_\pm$. In the case $\kappa_\pm = 0$, oscillations on $\theta$ are transported during the time, and no smoothness can be expected on this term. Homogenization in this context is known (see \cite{BrBuGJLa}). In this article, we focus on the case $\kappa_\pm>0$. The temperature equation is then partially parabolic, so we expect $\theta$ to be immediately smooth. Another quantity that we have not yet mentioned here but which is essential for the transition to the limit in \eqref{eq:cmeso}--\eqref{eq:energmeso} is the Cauchy stress $\sigma$, defined as
\begin{equation}\label{eq:defsigma}
\sigma = \mu(c)\partial_x u - R(c)\rho\theta.
\end{equation}
Since the Cauchy stress is the sum of the forces acting on a fluid, it seems natural, by the reciprocal action theorem, that $\sigma$ should be regular. In fact, we can deduce from the system \eqref{eq:cmeso}--\eqref{eq:energmeso} that
\begin{equation}\label{eq:sigma}
\partial_t \sigma + u\partial_x \sigma - \mu(c)\partial_x \left(\frac{\partial_x \sigma}{\rho}\right) = - \gamma(c)\sigma\partial_x u - (\gamma(c)-1)\partial_x(\kappa \partial_x\theta).
\end{equation}
As $\mu(c)$ is transported by $u$, thus $\sigma$ is expected to be smooth. Even in the barotropic case, Cauchy stress can provide interesting information in the analysis of Navier-Stokes equations (interested readers may wish to read some comments by D. Serre in \cite{Serre1991}). D. Hoff frequently employs this technique in his research on functional spaces that are situated between strong and weak solutions \cite{Ho1, Ho2, Hoff1986}. Recent progress has been made thanks to this quantity in the study of weak Navier-Stokes solutions for ideal gases in three dimensions \cite{HuLi18}, and the Cauchy stress has been identified as a key factor in the process of homogenization when $\kappa_\pm = 0$ \cite{BrBuGJLa}. 
The founding article concerning the existence of global solutions to the non-oscillatory system \eqref{eq:cont}--\eqref{eq:energ} is undoubtedly that of Kazhikhov and Shelukhin \cite{KaSh}, for strong solutions. This result was first improved in \cite{Ka}, by introducing the entropy $s$ defined in \eqref{eq:entropy}. Hoff found some a priori estimates using initial conditions very similar to \eqref{eq:debmidinitcond}--\eqref{eq:fininitcond} \cite{Ho92}. His analysis allows to deal with discontinuous densities, for data close to equilibrium, introducing some new energy linked with the spatial derivative of the Cauchy stress $\sigma$. We can cite here the remarkable result of \cite{ChHoTr}, giving some bounds for $t\in [0,+\infty[$ in the framework of Hoff but without condition on the smallness of the data. Jinkai Li recently proved a result enabling evanescent densities, also skillfully using the Cauchy stress \cite{Li2}.

The study of the Navier-Stokes system with heat-conduction and with oscillatory coefficients was first done by Amosov and Zlotnik \cite{AZ92, AZ96a, AZ96b, AZ97, AZ01}, from a Lagrangian point of view. Roughly speaking, one of the final main ideas of Kazhikhov and Shelukhin was to find a bound on $\partial_x e$. However, if $\cv$ is oscillating and $\theta$ is regular, such a bound grows very quickly when the frequency of the oscillations is high. Amosov and Zlotnik followed the reasoning of Kazhikhov and Shelukhin as long as their method was compatible with the presence of oscillatory coefficients. By following this approach, they succeeded in applying parabolic regularity results to obtain bounds on the spatial derivatives of temperature and velocity. Their development was based on initial data that was highly irregular. However, the adaptation of their work to the Eulerian case is arduous, and moreover, it appears to be incapable of characterising nonlinear density function limits, a problem which will be addressed in Section \ref{sec:homog}.
More recently, Hillairet proposed a strategy for adapting the homogenization results obtained with Bresch in the Eulerian barotropic case \cite{BrHi, BrHi2}. This approach enabled the characterisation of the limits of nonlinear functions in density, in cases where the coefficients are functions of density (see \cite{Hi2}). He found that the macroscopic quantities $(\alpha_\pm, \rho_\pm, u,\theta)$ are solution of the equation
\begin{empheq}[left=\empheqlbrace]{alignat=1}
\partial_t \alpha_+ + u\partial_x\alpha_+ &= \frac{\alpha_+\alpha_-}{\mu_{eff}}(R_+\rho_+\theta - R_-\rho_-\theta)\label{eq:cmacrotemp}\\
\partial_t (\alpha_+\rho_+) + \partial_x (\alpha_+\rho_+ u) &= 0
\label{eq:contmacrotemp}\\
\partial_t (\rho u) + \partial_x (\rho u^2) &=\partial_x\sigma
\label{eq:momtsmacrotemp}\\
\partial_t (\rho\cv_{eff}\theta) + \partial_x(\rho\cv_{eff} \theta u) &= \sigma\partial_x u +\partial_x (\kappa_{eff}\partial_x\theta),
\label{eq:energmacrotemp}
\end{empheq}
where
\begin{equation}
\mu_{eff} = \dfrac{1}{\dfrac{\alpha_+}{\mu_+} + \dfrac{\alpha_-}{\mu_-}}, \quad \kappa_{eff} = \dfrac{1}{\dfrac{\alpha_+}{\kappa_+} + \dfrac{\alpha_-}{\kappa_-}}, \quad \cv_{eff} = \frac{\alpha_+\rho_+\cv_+ + \alpha_-\rho_-\cv_-}{\alpha_+\rho_+ + \alpha_-\rho_-},
\end{equation}
\begin{equation}
\rho = \alpha_+\rho_+ + \alpha_-\rho_-, \quad p_{eff} = \mu_{eff}\left(\frac{\alpha_+R_+\rho_+}{\mu_+} + \frac{\alpha_-R_-\rho_-}{\mu_-}\right) \theta
\end{equation}
and
\begin{equation}
\sigma = \mu_{eff} \partial_x u - p_{eff}.
\end{equation}

Nevertheless, Hillairet's article exclusively provides estimates for short time, operating under the assumption that the energy of the system is fairly small.
The objective of this paper is to establish, for the first time, the limit of mesoscopic mixing of a mixture of two heat-conducting ideal gases towards macroscopic mixing, without assumptions of smallness on the initial data and over an arbitrary time interval. The paper's originality lies in the methodology of obtaining uniform bounds in the oscillation parameter $\varepsilon$.

\section{Main results}
Of course, the initial discussion that we had on the regularity of unknowns is heuristic. Nevertheless, it provides some indication of the solution framework. We then choose some regular initial velocity $u_0$ and temperature $\theta_0$, while the initial density $\rho_0$ and the initial color function $c_0$ are oscillatory. More precisely, the framework that we choose is the following: 

\begin{definition}\label{def:alaHoff}
Assume that $\kappa_\pm> 0$. We call a global "à la Hoff" solution of  \eqref{eq:cmeso}--\eqref{eq:energmeso} any weak solution $(c,\rho,u,\theta)\in L^2_{loc}(\R_+\times\T)^4$ of $(NS_\kappa)$ that also satisfies, for all $T>0$, 
\begin{equation}\label{eq:defc}
\text{for almost all }(t,x)\in [0,T]\times\T,\quad 0\leq c(t,x)\leq 1.
\end{equation}
\begin{equation}\label{eq:def12}
\text{for almost all }(t,x)\in[0,T]\times\T,\quad \underline{\rho}\leq \rho(t,x) \leq \overline{\rho},
\end{equation}
\begin{equation}\label{eq:def13}
\text{for almost all }(t,x)\in [0,T]\times\T,\quad \underline{\theta}\leq \theta(t,x) \leq \overline{\theta},
\end{equation}
\begin{equation}\label{eq:def14}
\begin{split}
&\sup_{0\leq t\leq T}\int_0^1\vert\sigma\vert^3 + \sup_{0\leq t\leq T}\int_0^1\vert\partial_x u\vert^3 + \sup_{0\leq t\leq T}\int_0^1 \vert\partial_x\theta\vert^2 + \int_0^T\int_0^1 \vert\partial_x\sigma\vert^2 + \int_0^T\int_0^1 \vert\partial_x(\kappa\partial_x\theta)\vert^2 \\&+ \int_0^T\int_0^1\vert\partial_t u\vert^2+ \int_0^T\int_0^1\vert\partial_t \theta\vert^2 + \int_0^T \Vert\sigma\Vert_\infty^3 + \int_0^T \Vert \partial_x u\Vert_\infty^3 + \sup_{[0,T]\times\T} \vert u\vert^2 \leq C_1.
\end{split}
\end{equation}
for some $\underline{\rho},\overline{\rho}, \underline{\theta},\overline{\theta}, C_1>0$ that may depend on $T$.
\end{definition}
The main novelty in this article is the establishment of the existence of global "à la Hoff" solutions to the Navier-Stokes system with heat-conduction and with $\underline{\rho},\overline{\rho},\underline{\theta},\overline{\theta}, C_0$ that do not depend on the oscillation frequency of $c_0$ and $\rho_0$. This is stated in the following theorem.

\begin{theorem}\label{thm:bounds} Let $(c_0,\rho_0,\theta_0,u_0)\in L^\infty(\T)^3\times L^2(\T)$.
Assume moreover that there exists some $\overline{\rho_0}>\underline{\rho_0}>0$, $\overline{\theta_0}>\underline{\theta_0}>0$, $C_0>0$ such that
\begin{equation}\label{eq:debinitcond}
\text{for almost all }x\in\T,\quad 0\leq c_0(x)\leq 1,
\end{equation}
\begin{equation}\label{eq:debmidinitcond}
\text{for almost all }x\in\T,\quad \underline{\rho_0} \leq \rho_0(x) \leq \overline{\rho_0},
\end{equation}
\begin{equation}\label{eq:finmidinitcond}
\text{for almost all }x\in\T,\quad \underline{\theta_0}\leq \theta_0(x)\leq \overline{\theta_0}, 
\end{equation}
\begin{equation}\label{eq:fininitcond}
\int_0^1\rho_0 u_0^2 + \int_0^1 \vert\partial_x u_0\vert^3 + \int_0^1 \vert\partial_x\theta_0\vert^2 \leq C_0.
\end{equation}
Then the system \eqref{eq:cmeso}--\eqref{eq:energmeso} with initial condition $(c_0,\rho_0,\theta_0,u_0)$ admits a unique "à la Hoff" solution. Moreover, the bounds of Definition \ref{def:alaHoff} can be chosen to depend only on $\mu_\pm,\gamma_\pm,\cv_\pm,\kappa_\pm$ and $\underline{\rho_0},\overline{\rho_0},\underline{\theta_0},\overline{\theta_0}, C_0$.
\end{theorem}

The combination of the $L^3$ norm and the $L^2$ norm in \eqref{eq:def14} may appear surprising, however the $L^3$-framework appears to be suitable for combining the estimates on $\sigma$ and $\theta$ (see Section \ref{sec:heur}). A similar idea was used in the barotropic case to relax the assumptions on the integrability of the density (see \cite{BrBu23}).

Thanks to the uniform bounds of Theorem \ref{thm:bounds}, we can then pass to the limit, by letting the initial conditions $\rho_0$, $c_0$ oscillate faster and faster. More precisely, for each $\varepsilon>0$, $\rho_0^\varepsilon$ is assumed to be of the form
\begin{equation}\label{eq:hicham}
\rho_0^\varepsilon = c_0^\varepsilon \rho_{0,+}^\varepsilon + (1-c_0^\varepsilon)\rho_{0,-}^\varepsilon.
\end{equation}
where there exist $\rho_{0,\pm}\in L^2(\T)$ such that
\begin{equation}\label{eq:cecile}
\rho_{0,\pm}^\varepsilon\tend{\varepsilon}{0} \rho_{0,\pm},
\end{equation}
so the initial macroscopic density is of the form
\begin{equation}\label{romain}
\rho_0 = \alpha_0\rho_{0,+} + (1-\alpha_0)\rho_{0,-},
\end{equation}
with $\alpha_0$ the initial volume fraction of the fluid $+$. This assumption is consistent with the modeling proposed above. We then obtain the main result:

\begin{theorem}\label{thm:main}
Let $T>0$ be fixed. Consider the sequences $(c_0^\varepsilon)_{\varepsilon>0},(\rho_{0,\pm}^\varepsilon)_{\varepsilon>0}\subset L^\infty(\T)$, $(u_0^\varepsilon)_{\varepsilon>0},(\theta_0^\varepsilon)_{\varepsilon>0}\subset H^1(\T)$ and some functions $\alpha_0,{\rho_{0,\pm}}\in L^\infty(\T)$, $u_0,\theta_0\in H^1(\T)$. Assume that $(c_0^\varepsilon), \alpha_0$ satisfy \eqref{eq:debinitcond}, $(\rho_{0,\pm}^\varepsilon),\rho_{0,\pm}$ satisfy \eqref{eq:debmidinitcond}, and that $(u_0^\varepsilon,\theta_0^\varepsilon), (u_0,\theta_0)$ satisfy \eqref{eq:finmidinitcond} and \eqref{eq:fininitcond} with $\underline{\rho_0},\overline{\rho_0},\underline{\theta_0},\overline{\theta_0}$ and $C_0$ that do not depend on $\varepsilon$. Assume moreover that \eqref{eq:hicham} and \eqref{eq:cecile} hold, and 
\begin{equation}
c_0^\varepsilon \tend{\varepsilon}{0} \alpha_0\quad \text{in }L^\infty(\T)-\star,
\end{equation}
\begin{equation}
\rho_0^\varepsilon \tend{\varepsilon}{0}\rho_0\quad \text{in }L^\infty(\T)-\star,
\end{equation}
\begin{equation}
u_0^\varepsilon,\theta_0^\varepsilon \tendf{\varepsilon}{0} u_0,\theta_0\quad \text{in }H^1(\T),
\end{equation}
where $\rho_0^\varepsilon$ (resp. $\rho_0$) is defined like in \eqref{eq:hicham} (resp. \eqref{eq:cecile}).
For all $\varepsilon>0$, consider $(c^\varepsilon,\rho^\varepsilon,u^\varepsilon,\theta^\varepsilon)$ the "à la Hoff" solution of \eqref{eq:cmeso}--\eqref{eq:energmeso} with initial condition $(c_0^\varepsilon,\rho_0^\varepsilon,u_0^\varepsilon,\theta_0^\varepsilon)$.
Then, there exists some $\alpha_\pm,\rho_\pm \in L^\infty([0,T]\times\T)$ and $u,\theta\in L^2(0,T,H^1(\T))$ such that $\alpha_+ + \alpha_- = 1$ and, up to a subsequence,
\begin{equation}
c^\varepsilon,\rho^\varepsilon \tend{\varepsilon}{0} \alpha_+,\alpha_+\rho_+ + \alpha_-\rho_-\quad \text{in }L^\infty(\T)-\star,
\end{equation}
\begin{equation}
u^\varepsilon,\theta^\varepsilon \tend{\varepsilon}{0} u,\theta\quad \text{in }L^2(0,T,L^2(\T)).
\end{equation}
Moreover, $(\alpha_\pm,\rho_\pm,u,\theta)$ is a solution to \eqref{eq:cmacrotemp}--\eqref{eq:energmacrotemp} with initial condition $(\alpha_0, 1-\alpha_0,\rho_{0,\pm},u_0,\theta_0)$.
\end{theorem}

Section \ref{sec:estimates} is devoted to the proof of Theorem \ref{thm:bounds}. We provide only the uniform bounds here, as the system \eqref{eq:cmeso}--\eqref{eq:energmeso} falls within the theory of partially dissipative hyperbolic systems \cite{Kaw,KS,Se10}, which guarantees the local existence of smooth solutions (see the Appendix \ref{app:justif} for more details). The first subsection is heavily inspired by ideas of \cite{KaSh, Ka}, which are simply adapted to Eulerian formulation and oscillatory coefficients. The second subsection outlines the process of obtaining new bounds and justifies the use of the $L^3$-framework. The complete proof of the new bounds is presented in the third subsection. Section \ref{sec:homog} is focused on the proof of Theorem \ref{thm:main}. The first and second subsections establish strong convergence for the velocity and the temperature, and introduce a Young measure in order to encode the limit of nonlinear terms on the density and the color function. This part closely follows the analysis in \cite{Hi2}. Finally, in the third subsection, the Young's measure is characterized, using arguments from \cite{BrHi, BrHi2, BrBuGJLa}.

In all that follows, $\mu_\pm>0$, $\cv_\pm>0$, $\gamma_\pm>1$, $\kappa_\pm$ and $R_\pm$ are fixed once and for all.

\section{A priori estimates}\label{sec:estimates}
The reader is invited to consult Appendix \ref{app:not} for the definitions of the notations $\Vert \cdot\Vert_p$, $D_t$, $D_t^{-1}$, $R_t$, $\partial_x^{-1}$, $f_{\max}$ and $f_{\min}$.
\subsection{Some classical estimates}
This section is divided as follows. First, we give the proof of several basic estimates. Informally, we adapt some part of the method in \cite{KaSh,Ka}, as has already been done in \cite{AZ92} in a Lagrangian framework. Then, we present the main ideas in order to go further. The $L^3$ framework for $\sigma$ and $\theta$ is well-adapted because using the bounds obtained in the first part we are able to bound $L^4$ norm quantities in $\sigma$, $\theta$ by $L^3$ norm quantities in $\sigma$, $\theta$ (see Figure \ref{fig:schempreuve} and Proposition \ref{prop:fond} for a rigorous statement). We prove Theorem \ref{thm:bounds}. In the sequel, let us consider $(c_0,\rho_0,u_0,\theta_0)\in H^1(\T)^4$ satisfying \eqref{eq:debinitcond}--\eqref{eq:fininitcond} with some $\underline{\rho_0},\overline{\rho_0},\underline{\theta_0},\overline{\theta_0},C_0>0$. Let us fix $T>0$. We denote by $(c,\rho,u,\theta)$ the strong solution of \eqref{eq:cmeso}--\eqref{eq:energmeso} with initial condition $(c_0,\rho_0,\theta_0,u_0)$.  All the following bounds will implicitly depend on $\mu_\pm,\cv_\pm,\gamma_\pm,\kappa_\pm$.

\begin{prop}
\label{prop:1} Let us denote by
\begin{equation}
{\cal M} = \int_0^1\rho_0,\quad {\cal E} = \int_0^1\rho_0 E_0.
\end{equation}
Then $0<{\cal M}\leq \overline{\rho_0}$ and $0<{\cal E}\leq \overline{\rho_0}\cv_{\max}\overline{\theta_0} + \mu_{\max}C_0/2$, and
\begin{equation}
   \text{for all }t\in[0,T],\quad\cM = \int_0^1\rho(t)\label{eq:26},
\end{equation}
\begin{equation}
    \text{for all }t\in[0,T],\quad\cE= \int_0^1 \rho E(t), \label{eq:27}
\end{equation}
\begin{equation}
    \sup_{0\leq t\leq T}\int_0^1 p(t) \leq (\gamma_{\max}-1)\cE,
\label{eq:28}
\end{equation}
\begin{equation}
    \sup_{0\leq t\leq T}\int_0^1 \frac{\rho u^2}{2}(t) \leq \cE,
\label{eq:29}
\end{equation}
\begin{equation}
    \sup_{0\leq t\leq T}\int_0^1 \vert\rho u\vert(t)\leq  \sqrt{2\cM\cE}.
\label{eq:30}
\end{equation}
\end{prop}
\begin{proof}
Integrating \eqref{eq:contmeso} (resp. \eqref{eq:momtsmeso}) over the torus gives \eqref{eq:26} (resp. \eqref{eq:27}). We deduce \eqref{eq:28}, \eqref{eq:29} from \eqref{eq:27} and the positivity of $\rho$ and $p$. Finally, using Cauchy-Schwartz inequality,
\begin{equation*}
    \int_0^1 \vert\rho u\vert = \sqrt{2}\int_0^1 \sqrt{\rho}\sqrt{\rho}\frac{\vert u\vert}{\sqrt{2}}\leq \sqrt{2}\sqrt{\int_0^1\rho}\sqrt{\int_0^1 \frac{\rho u^2}{2}}\leq  \sqrt{2\cM\cE}.
\end{equation*}
\end{proof}
\begin{remark}
Observe that if $v$ is a constant, the change of unknowns
\begin{equation*}
    (c(t,x),\rho(t,x),u(t,x),E(t,x)) \rightarrow (c(t,x-vt),\rho(t,x-vt), u(t,x-vt)+v, E(t,x-vt))
\end{equation*}
always gives a solution of \eqref{eq:cmeso}--\eqref{eq:energmeso} (it is the Galilean invariance principle). Choosing
\begin{equation*}
    v = -\frac{1}{\cM}\int_0^1\rho_0 u_0
\end{equation*}
we may assume without loss of generality
\begin{equation*}
    \int_0^1\rho_0u_0 = 0.
\end{equation*}
Then, integrating \eqref{eq:contmeso} over the torus, we get
\begin{equation}\label{eq:31}
    \text{for all }t\in[0,T],\quad \int_0^1\rho u(t) = \int_0^1\rho_0 u_0 = 0.
\end{equation}
\end{remark}

The next proposition concerns entropy.
\begin{prop}
\label{prop:2}
There exists some constant $H_1 = H_1(\underline{\rho_0}, \overline{\rho_0}, \underline{\theta_0}, \overline{\theta_0})>0$ such that
\begin{equation}\label{eq:32}
\int_0^T\int_0^1\frac{(\partial_x u)^2}{\theta}+\int_0^T\int_0^1 \frac{(\partial_x \theta)^2}{\theta^2} \leq H_1.
\end{equation}
\end{prop}
\begin{proof}
Considering the nonnegative function $h:]0,+\infty[\rightarrow \R$ defined as
\noindent

\vspace{0.5cm}
\begin{minipage}{0.4\textwidth}
\begin{equation*}
\text{for all }x \in ]0,+\infty[, \quad h(x) = x - 1 - \ln(x),
\end{equation*}
\end{minipage}
\hfill
\begin{minipage}{0.55\textwidth}
\begin{minipage}{0.4\linewidth}
\hspace{2.5em}
\begin{tikzpicture}
\begin{axis}[
    height=3cm,
    width=\linewidth,
    xmin=0,
    xmax=8,
    axis lines=middle,
    xtick={1},
    ytick=\empty
]
\addplot[mark=none, line width=1pt] table {function_data.txt};
\end{axis}
\end{tikzpicture}
\end{minipage}
\begin{minipage}{0.55\linewidth}
\small\textbf{Figure 1:} Representation of the function \(h\)
\end{minipage}

\end{minipage}

\noindent we obtain from \eqref{eq:entropy}
\begin{equation}\label{eq:33}
\rho \cv h(\theta) +\rho R h(1/\rho) + \rho s = \cv\rho\theta - \cv\gamma\rho + R.
\end{equation}
Note that
\begin{equation}\label{eq:massccons}
\frac{d}{dt}\int_0^1\rho c = 0.
\end{equation}
As a consequence, differentiating \eqref{eq:33} in time then integrating over the torus, using \eqref{eq:entropymeso} then \eqref{eq:26},\eqref{eq:massccons}, we obtain

\begin{equation}\label{eq:34}
\frac{d}{dt}\int_0^1 (\rho \cv h(\theta) + \rho R h(1/\rho)) + \int_0^1 \mu\frac{(\partial_x u)^2}{\theta} + \int_0^1\kappa\frac{(\partial_x \theta)^2}{\theta^2} = \frac{d}{dt}\int_0^1 \rho \cv\theta + \frac{d}{dt}\int_0^1 R 
\end{equation}
hence, from \eqref{eq:27},
\begin{equation}
\frac{d}{dt}\int_0^1 (\rho \cv h(\theta) + \rho R h(1/\rho)) + \frac{d}{dt}\int_0^1 \frac{\rho u^2}{2}+ \int_0^1 \mu\frac{(\partial_x u)^2}{\theta} + \int_0^1\kappa\frac{(\partial_x \theta)^2}{\theta^2} = \frac{d}{dt}\int_0^1 R .
\end{equation}
Finally, integrating \eqref{eq:34} on $[0,T]$, we get
\begin{align}
\int_0^1 (\rho \cv h(\theta) + \rho R h(1/\rho))(T) &+ \int_0^1\frac{\rho u^2}{2}(T) + \int_0^T\int_0^1 \mu \frac{(\partial_x u)^2}{\theta} + \int_0^T\int_0^1\kappa\frac{(\partial_x\theta)^2}{\theta^2} \nonumber\\&= \int_0^1 \frac{\rho_0 u_0^2}{2} +\int_0^1 R - \int_0^1 R_0 + \int_0^1 (\rho_0 \cv h(\theta_0) + \rho_0 R h(1/\rho_0)) \nonumber\\&\leq {\cal E} + R_{\max} + \cv_{\max} {\cal M} \sup_{[\underline{\theta_0},\overline{\theta_0}]} h  + R_{\max} {\cal M} \sup_{[1/\overline{\rho_0},1/\underline{\rho_0}]} h. \nonumber\end{align}
and in particular the result, dividing this last inequality by $\min(\mu_{\min},\kappa_{\min})$, and using the fact that $h\geq 0$.
\end{proof}
The next bound provides information on the Cauchy stress. It has been widely used for both barotropic and non-isentropic (see \cite{Serre1991}).
\begin{prop}\label{prop:3}
There exists some $H_2 = H_2({\cal E}, {\cal M}, T)>0$ such that
\begin{equation}\label{eq:35}
\text{for almost all }(t,x)\in [0,T]\times \T,\quad \vert D_t^{-1}(\sigma)(t,x)\vert\leq H_2.
\end{equation} 
\end{prop}
\begin{proof}
Applying $\partial_x^{-1}$ to \eqref{eq:momtsmeso}, we get using \eqref{eq:25}
\begin{equation}\label{eq:36}
\partial_t \partial_x^{-1}(\rho u) + \partial_x^{-1}\partial_x(\rho u^2) = \sigma - \int_0^1 \sigma.
\end{equation}
Moreover, due to \eqref{eq:cmeso},
\begin{equation}
\int_0^1\sigma = \int_0^1\mu(c)\partial_x u - \int_0^1 p = \frac{d}{dt}\int_0^1 \mu(c) - \int_0^1 p = D_t\left(\int_0^1\mu - \int_0^t\int_0^1 p\right).
\end{equation}
And, due to \eqref{eq:31} and \eqref{eq:25},
\begin{equation}\label{eq:37}
\partial_x^{-1}\partial_x(\rho u^2) = \rho u^2 - \int_0^1\rho u^2 = u\partial_x\partial_x^{-1}(\rho u) - \int_0^1\rho u^2.
\end{equation}
Then \eqref{eq:37} in \eqref{eq:36} gives
\begin{equation}\label{eq:38}
D_t \partial_x^{-1}(\rho u) = \sigma - D_t\left(\int_0^1\mu - \int_0^t\int_0^1 p - \int_0^t\int_0^1\rho u^2\right)
\end{equation}
Finally, applying $D_t^{-1}$ to \eqref{eq:38}, we get using \eqref{eq:23},
\begin{equation*}
D_t^{-1}\sigma = D_t^{-1} D_t \partial_x^{-1}(\rho u) + \int_0^1\mu - \int_0^1\mu_0 - \int_0^t\int_0^1 p -\int_0^t\int_0^1\rho u^2,
\end{equation*}
thus from \eqref{eq:20}, \eqref{eq:21}, \eqref{eq:28}, \eqref{eq:29}, \eqref{eq:30}, \eqref{eq:31} and \eqref{eq:24},
\begin{align}
\vert D_t^{-1} \sigma\vert &\leq \vert \partial_x^{-1}(\rho u)\vert + \vert R_t\partial_x^{-1}(\rho u)\vert + \left\vert\int_0^1 \mu - \int_0^1\mu_0\right\vert + \left\vert\int_0^t\int_0^1 p\right\vert + \left\vert\int_0^t\int_0^1\rho u^2\right\vert \nonumber\\&\leq \int_0^1 \vert \rho u\vert + \sup_{x\in\T}\vert\partial_x^{-1}(\rho_0 u_0)(x)\vert + \mu_{\max} + T(\gamma_{\max}+1){\cal E} \nonumber\\&\leq 2\sqrt{2{\cal M}{\cal E}} + \mu_{\max} + T(\gamma_{\max}+1){\cal E}. \nonumber\end{align}
\end{proof}

From Proposition \ref{prop:3}, we deduce an upper bound on $\rho$.
\begin{prop}\label{prop:4}
There exists some $\overline{\rho} = \overline{\rho}({\cal E}, {\cal M}, \overline{\rho_0},T)>0$ such that
\begin{equation}\label{eq:39}
\text{for almost all }(t,x)\in [0,T]\times \T,\quad \rho(t,x)\leq \overline{\rho}.
\end{equation} 
\end{prop}
\begin{proof}
From \eqref{eq:contmeso} and \eqref{eq:defsigma} we obtain
\begin{equation}\label{eq:40}
D_t\ln\rho = -\partial_x u = \frac{-\sigma - p}{\mu}.
\end{equation}
Hence, applying $D_t^{-1}$ to \eqref{eq:40} then using \eqref{eq:20}, \eqref{eq:21}, \eqref{eq:rigolo} and \eqref{eq:35},
\begin{align}
\ln \rho &= R_t\ln\rho -\frac{D_t^{-1}\sigma}{\mu} - \frac{D_t^{-1} p}{\mu}\label{eq:41}
\\&\leq \ln \overline{\rho_0} + \frac{\vert D_t^{-1}\sigma\vert}{\mu} \nonumber\\&\leq \ln \overline{\rho_0} + \frac{H_2}{\mu_{\min}} \nonumber\end{align}
because $D_t^{-1} p \geq 0$. Finally, passing to the exponential,
\begin{equation*}
\text{for almost all }(t,x)\in [0,T]\times\T,\quad\rho(t,x) \leq \overline{\rho_0}\exp{(H_2/\mu_{\min})}.
\end{equation*}
\end{proof}
A link between the lower bound of $\rho$ and the upper bound on $\theta$ is then obtained.
\begin{prop}
\label{prop:6}
There exists some $H_3 = H_3({\cal E}, \underline{\rho_0}, \overline{\rho_0},T)>0$ such that
\begin{equation}\label{eq:48}
\text{for almost all }(t,x)\in [0,T]\times \T,\quad 1/\rho(t,x) \leq H_3\left(1+\int_0^t \Vert\theta\Vert_\infty\right).
\end{equation}
\end{prop}
\begin{proof}
Applying the exponential function to \eqref{eq:41}, we obtain
\begin{equation}\label{eq:49}
\rho \exp(D_t^{-1} p/\mu) = \exp(R_t\ln \rho - D_t^{-1}\sigma/\mu) =: B.
\end{equation}
Note that, from \eqref{eq:21} and \eqref{eq:35} we get
\begin{equation}\label{eq:50}
\text{for almost all }(t,x)\in [0,T]\times \T,\quad 0 < \underline{B}\leq B(t,x)\leq \overline{B}
\end{equation}
where
\begin{equation*}
\underline{B} = \underline{\rho_0}\exp(- H_2/\mu_{\min}),\quad \overline{B}=\overline{\rho_0}\exp(H_2/\mu_{\min}).
\end{equation*}
Multiplying now \eqref{eq:49} by $R\theta/\mu$, we obtain using \eqref{eq:cmeso}
\begin{equation*}
D_t \exp(D_t^{-1} p/\mu) = B R\theta/\mu
\end{equation*}
then applying $D_t^{-1}$,
\begin{equation}\label{eq:51}
\exp(D_t^{-1} p/\mu) = 1 + D_t^{-1}(B R\theta/\mu),
\end{equation}
because $D_t^{-1} p(0,\cdot) = 0$.
Combining \eqref{eq:51} with \eqref{eq:49}, we get
\begin{equation}\label{eq:52}
B/\rho = 1 + D_t^{-1}(BR\theta/\mu).
\end{equation}
Dividing \eqref{eq:52} by $B$, we finally get using \eqref{eq:50}, the non negativity of $D_t^{-1}$ and \eqref{eq:22},
\begin{equation*}
1/\rho = 1/B + D_t^{-1}(B R\theta/\mu)/B \leq \frac{1}{\underline{B}}\left(1 + (\overline{B}R_{\max}/\mu_{\min}) \int_0^t\Vert\theta\Vert_\infty\right)\leq H_3\left(1 + \int_0^t \Vert \theta\Vert_\infty\right)
\end{equation*}
for all $t\in[0,T]$ and with
\begin{equation*}
H_3 = \max((\overline{B}R_{\max}/\mu_{\min})/\underline{B}, 1/\underline{B}).
\end{equation*}
\end{proof}
\begin{prop}\label{prop:11}
There exists some $H_4 = H_4({\cal E}, H_1,\underline{\rho_0}, \overline{\rho_0}, T)>0$ such that
\begin{equation}\label{eq:69}
\int_0^T\Vert \theta\Vert_{\infty} \leq H_4.
\end{equation}
\end{prop}
\begin{proof}
Let us start by the equality, for almost all $x,y\in \T$, $t\in [0,T]$,
\begin{equation}\label{eq:70}
\sqrt{\theta}(t,x) = \sqrt{\theta}(t, y) + \int_y^x \frac{\partial_z \theta}{2\sqrt{\theta}}(t, z) dz.
\end{equation}
Observe that from \eqref{eq:48} we have, for almost all $t\in[0,T]$,
\begin{equation*}
\left\vert\frac{\partial_z \theta}{\sqrt{\theta}}\right\vert = \left\vert\frac{\partial_z \theta}{\theta}\right\vert\sqrt{\rho\theta}\frac{1}{\sqrt{\rho}}\leq \left\vert\frac{\partial_z \theta}{\theta}\right\vert\sqrt{\rho\theta} \sqrt{H_3}\sqrt{1+\int_0^t\Vert\theta\Vert_\infty}.
\end{equation*}
Thus by Hölder's inequality and \eqref{eq:28}
\begin{equation}\label{eq:71}
\left\vert\int_y^x \frac{\partial_z \theta}{2\sqrt{\theta}}\right\vert \leq \frac{1}{2}\sqrt{\frac{{\cal E}H_3}{\cv_{\min}}}\sqrt{\int_0^1 \frac{(\partial_z \theta)^2}{\theta^2}} \sqrt{1+\int_0^t\Vert\theta\Vert_\infty}.
\end{equation}
Then, taking the square in \eqref{eq:70}, using Young's inequality then \eqref{eq:71}, we obtain
\begin{align}
\theta(t,x) \leq 2\theta(t,y) + \frac{1}{2}\frac{{\cal E}H_3}{\cv_{\min}}\int_0^1\frac{(\partial_z\theta)^2}{\theta^2}\left(1+\int_0^t\Vert\theta\Vert_\infty\right). \nonumber
\end{align}
Multiplying this last inequality by $\rho(t,y)$, then integrating over the torus in $y$, we obtain from \eqref{eq:26} and \eqref{eq:28}
\begin{equation}\label{eq:72}
{\cal M}\theta(t,x) \leq \frac{2{\cal E}}{\cv_{\min}} + \frac{1}{2}\frac{{\cal M}{\cal E}H_3}{\cv_{\min}}\int_0^1 \frac{(\partial_z \theta)^2}{\theta^2}\left(1 + \int_0^t\Vert \theta\Vert_\infty\right).
\end{equation}
Finally, dividing \eqref{eq:72} by $\cal M$ and passing to the sup on space at the left hand side, we obtain
\begin{equation*}
\Vert \theta\Vert_\infty(t) \leq \frac{2{\cal E}}{\cv_{\min} {\cal M}} + \frac{1}{2}\frac{{\cal E}H_3}{\cv_{\min}}\int_0^1 \frac{(\partial_z\theta)^2}{\theta^2}\left(1+\int_0^t\Vert\theta\Vert_\infty\right).
\end{equation*}
Using Grönwall's lemma then \eqref{eq:32}, we get
\begin{equation}\label{eq:73}
1 + \int_0^t\Vert \theta\Vert_\infty \leq  \left(1 + \frac{2{\cal E}}{\cv_{\min} {\cal M}} T\right)\exp\left(\frac{{\cal E} H_3 H_1}{2\cv_{\min}}\right)
\end{equation}
thus the result.
\end{proof}
Combining Proposition \ref{prop:6} and Proposition \ref{prop:11}, the proof of the lower bound on the density is straightforward:
\begin{prop}\label{prop:rholower}
There exists some $\underline{\rho} = \underline{\rho}({\cal E}, H_1, {\cal M}, \underline{\rho_0}, \overline{\rho_0}, T)>0$ such that
\begin{equation}\label{eq:rholower}
\text{for almost all }(t,x)\in [0,T]\times \T,\quad \rho(t,x)\geq \underline{\rho}.
\end{equation} 
\end{prop}

A lower bound is then set for the temperature.
\begin{prop}\label{prop:5}
There exists some $\underline{\theta} = \underline{\theta}({\cal E}, {\cal M}, \overline{\rho_0}, T, \underline{\theta_0})$ such that
\begin{equation*}
\text{for almost all }(t,x)\in [0,T]\times\T,\quad \theta(t,x)\geq \underline{\theta}.
\end{equation*}
\end{prop}
\begin{proof}
Let $\beta:\R\rightarrow\R$ be some convex and non increasing function. Multiplying \eqref{eq:thetameso} by $\beta'(\theta)$ we get using \eqref{eq:cmeso}
\begin{align}\label{eq:42}
\partial_t(\rho\cv\beta(\theta)) + \partial_x(\rho u\cv\beta(\theta)) &= (\sigma\partial_x u)\beta'(\theta) + (\partial_x(\kappa\partial_x\theta))\beta'(\theta).
\end{align}
But from \eqref{eq:defsigma},
\begin{equation}\label{eq:43}
\sigma\partial_x u = \frac{\sigma^2}{\mu} + \frac{\sigma R\rho\theta}{\mu} = \frac{(\sigma + R\rho\theta/2)^2}{\mu} - \frac{R^2\rho^2\theta^2}{4\mu}\geq \frac{-R^2\rho^2\theta^2}{4\mu}
\end{equation}
and $\beta'(\theta)\leq 0$, therefore \eqref{eq:42} and \eqref{eq:43} give
\begin{equation}\label{eq:44}
\partial_t(\rho\cv\beta(\theta)) + \partial_x(\rho u\cv\beta(\theta)) \leq -\frac{R^2\rho^2\theta^2\beta'(\theta)}{4\mu} + (\partial_x(\kappa\partial_x\theta))\beta'(\theta).
\end{equation} 
Integrating \eqref{eq:44} over the torus we get by integration by parts
\begin{equation*}
\frac{d}{dt}\int_0^1 \rho\cv\beta(\theta) \leq -\int_0^1 \frac{R^2\rho^2\theta^2\beta'(\theta)}{4\mu} - \int_0^1 \kappa(\partial_x\theta)^2\beta''(\theta) \leq -\int_0^1\frac{R^2\rho^2\theta^2\beta'(\theta)}{4\mu}
\end{equation*}
because $\beta''(\theta)\geq 0$. Choosing now $\beta :x\mapsto (1/x)^k$ for $k\in\N^*$, we obtain
\begin{equation}\label{eq:45}
\frac{d}{dt}\int_0^1\frac{\rho\cv}{\theta^k}\leq \frac{R_{\max}^2\overline{\rho}k}{4\mu_{\min}}\int_0^1\frac{\rho}{\theta^{k-1}} \leq \frac{R_{\max}^2\overline{\rho} k}{4\mu_{\min}\cv_{\min}^{1-1/k}}{\cal M}^{1/k}\left(\int_0^1\frac{\rho\cv}{\theta^k}\right)^{1-1/k}
\end{equation}
using Hölder's inequality then \eqref{eq:26}. As a consequence, multiplying \eqref{eq:45} by 
\begin{equation*}
\frac{1}{k}\left(\int_0^1 \frac{\rho\cv}{\theta^k}\right)^{1/k-1},
\end{equation*}
we get
\begin{equation}\label{eq:46}
\frac{d}{dt}\left(\int_0^1\frac{\rho\cv}{\theta^k}\right)^{1/k} \leq \frac{R_{\max}^2\overline{\rho}}{4\mu_{\min}\cv_{\min}^{1-1/k}}{\cal M}^{1/k}.
\end{equation}
Moreover, as $\rho>0$ and $\rho$ is regular, we have
\begin{equation*}
\left(\int_0^1 \frac{\rho\cv}{\theta^k}\right)^{1/k}\underset{k\rightarrow +\infty}{\rightarrow}\left\Vert \frac{1}{\theta} \right\Vert_\infty \quad\text{in }{\cal D}'(0,T),
\end{equation*}
hence, from \eqref{eq:46}
\begin{equation}\label{eq:47}
\frac{d}{dt}\left\Vert \frac{1}{\theta}\right\Vert_{\infty} \leq \frac{R_{\max}^2\overline{\rho}}{4\mu_{\min}\cv_{\min}}.
\end{equation}
Finally, integrating \eqref{eq:47} on $[0,t]$ for $t\in [0,T]$, we get
\begin{equation*}
\text{for almost all }(t,x)\in [0,T]\times \T,\quad\theta(t,x) \geq \frac{1}{1/\underline{\theta_0} + T R_{\max}^2\overline{\rho}/(4\mu_{\min}\cv_{\min})}.
\end{equation*}
\end{proof}

\subsection{Heuristic for obtaining stronger bounds}\label{sec:heur}

The key to continuing the analysis is to bound the Cauchy stress $\sigma$.
Let $\beta$ be some convex function. Multiplying \eqref{eq:sigma} by $\beta'(\sigma)/\mu$, then integrating over the torus, we get
\begin{equation}
\frac{d}{dt}\int_0^1 \frac{\beta(\sigma)}{\mu} + \int_0^1 \frac{(\partial_x\sigma)^2\beta''(\sigma)}{\rho} = \int_0^1\frac{\beta(\sigma)-\gamma\sigma\beta'(\sigma)}{\mu}\partial_x u -\int_0^1 \frac{\gamma-1}{\mu}\partial_x(\kappa\partial_x\theta)\beta'(\sigma).
\end{equation}
Similarly, multiplying \eqref{eq:thetameso} by $\beta'(\theta)$ we obtain
\begin{equation}
\frac{d}{dt}\int_0^1 \rho\cv\beta(\theta) + \int_0^1 (\partial_x\theta)^2\beta''(\theta) = \int_0^1\beta'(\theta)\sigma\partial_x u.
\end{equation}
For concreteness, let us choose $\beta(s) = \vert s\vert^k/k$ for some $k\in\N^*$. Then these last equations can be rewritten as
\begin{equation}\label{eq:ka}
\frac{d}{dt}\int_0^1 \frac{\vert\sigma\vert^k}{k\mu} + (k-1)\int_0^1\frac{(\partial_x\sigma)^2\vert\sigma\vert^{k-2}}{\rho} = \int_0^1 \frac{1-k\gamma}{k\mu}\vert\sigma\vert^k\partial_x u - \int_0^1\frac{\gamma-1}{\mu}\partial_x(\kappa\partial_x\theta)\sigma\vert\sigma\vert^{k-2}.
\end{equation}
\begin{equation}\label{eq:ki}
\frac{d}{dt}\int_0^1 \frac{\rho\cv\theta^k}{k} + (k-1)\int_0^1 (\partial_x\theta)^2 \theta^{k-2} = \int_0^1 \theta^{k-1}\sigma\partial_x u. 
\end{equation}
The last integral of the right hand side of \eqref{eq:ka} seems particularly annoying. The natural idea for bounding it is to start with some integration by parts, which is impossible due to the presence of the oscillatory coefficient $(\gamma-1)/\mu$, which depends on $c$. In the following, we therefore try to bound $\partial_x(\kappa\partial_x\theta)$ in $L^2(0,T,L^2(\T))$.  

Starting from \eqref{eq:thetameso}, we get
\begin{equation}
\partial_t(\partial_x \theta) + \partial_x(u\partial_x \theta) = \partial_x \dot{\theta} = \partial_x\left(\frac{\sigma\partial_x u}{\rho\cv}\right) +\partial_x\left(\frac{\partial_x(\kappa\partial_x\theta)}{\rho\cv}\right)
\end{equation}
then multiplying by $\kappa\partial_x\theta$ and integrating by parts (using the fact that $\kappa$ is transported by $u$), we get, using Young's inequality,
\begin{align}
\frac{1}{2}\frac{d}{dt}\int_0^1 \kappa(\partial_x\theta)^2 &+ \int_0^1 \frac{[\partial_x(\kappa\partial_x\theta)]^2}{\rho\cv} = -\int_0^1\frac{\kappa(\partial_x\theta)^2}{2}\partial_x u -\int_0^1 \frac{\sigma\partial_x u}{\rho\cv}\partial_x(\kappa\partial_x\theta) 
\\&\leq -\int_0^1 \frac{\kappa(\partial_x\theta)^2}{2}\partial_x u + \frac{1}{2\underline{\rho}\cv_{\min}}\int_0^1 \sigma^2(\partial_xu)^2 + \frac{1}{2}\int_0^1 \frac{[\partial_x(\kappa\partial_x\theta)]^2}{\rho\cv}.
\end{align}
thus
\begin{align}\label{eq:est1}
\frac{d}{dt}\int_0^1 \kappa(\partial_x\theta)^2 + \int_0^1 \frac{[\partial_x(\kappa\partial_x\theta)]^2}{\rho\cv} &\leq  \frac{1}{\underline{\rho}\cv_{\min}}\int_0^1 \sigma^2(\partial_x u)^2 -\int_0^1 \kappa(\partial_x\theta)^2\partial_xu.
\end{align}
Let us look at the first integral of the right hand side. This term is of degree $4$ in $\sigma,\partial_x u$. By comparing with \eqref{eq:ka},\eqref{eq:ki}, $k=3$ seems a good choice in order to close the estimates. From \eqref{eq:ka} and \eqref{eq:ki} we get
\begin{align}\label{eq:est2}
\frac{1}{3}\frac{d}{dt}\int_0^1 \frac{\vert \sigma\vert^3}{\mu} + 2\int_0^1 \frac{(\partial_x \sigma)^2\vert\sigma\vert}{\rho} &= \int_0^1 \frac{1/3-\gamma}{\mu}\vert\sigma\vert^3\partial_x u - \int_0^1\frac{\gamma-1}{\mu}\partial_x(\kappa\partial_x\theta)\sigma\vert\sigma\vert,
\end{align}
\begin{align}\label{eq:est3}
\frac{1}{3}\frac{d}{dt}\int_0^1\rho \cv\theta^3 + 2\int_0^1\kappa (\partial_x\theta)^2\theta &= \int_0^1\sigma(\partial_x u)\theta^2.
\end{align}
Informally, the main difficulty remains controlling terms of degree 4 in $\sigma, \partial_x u,\theta$ using terms of degree 3 in $\sigma, \partial_x u, \theta$. Figure \ref{fig:schempreuve} represents a strategy for achieving this. Figure \ref{fig:schempreuve} has to be understood as follows: from any given node, the corresponding term can be bounded by a linear combination of the terms associated with the nodes reached by the outgoing arrows. The prefix $\eta$ indicates that the associated coefficient can be taken as small as desired. For example, arrows $(a)$ indicate that for all $\eta>0$, there exists some $C_\eta = C_\eta(\eta,\overline{\rho_0},\underline{\rho_0}, \overline{\theta_0},\underline{\theta_0},C_0)$ such that
\begin{equation}
\left\vert\int_0^1\sigma^3\partial_x u\right\vert,\quad \left\vert\int_0^1\vert\sigma\vert^3\partial_x u\right\vert\leq \eta\int_0^1\frac{(\partial_x\sigma)^2\vert\sigma\vert}{\rho} + C_\eta\int_0^1\frac{\vert\sigma\vert^3}{\mu}.
\end{equation}
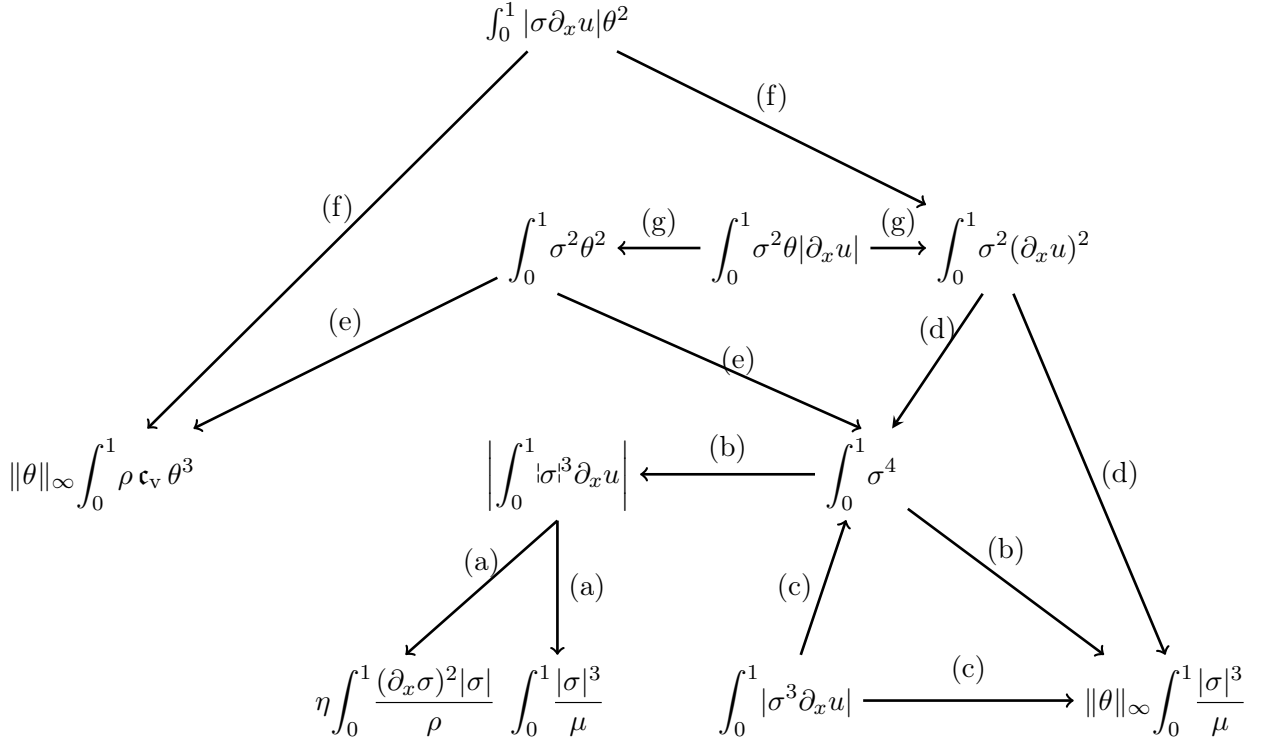
\begin{figure}[!ht]
\centering
\caption{Scheme of the proof}
\label{fig:schempreuve}
\begin{tikzpicture}
    \node (A) at (6,0) {$\displaystyle{\int_0^1}\sigma^2(\partial_x u)^2$};
    \node (K) at (3,0) {$\displaystyle{\int_0^1}\sigma^2\theta\vert\partial_x u\vert$};
    \node (B) at (0,0) {$\displaystyle{\int_0^1}\sigma^2\theta^2$};
    \node (C) at (4,-3) {$\displaystyle{\int_0^1}\sigma^4$};
    \node (E) at (8,-6) {$\Vert\theta\Vert_\infty\displaystyle{\int_0^1}\frac{\vert\sigma\vert^3}{\mu}$};
    \node (D) at (3,-6) {$\displaystyle{\int_0^1}\vert\sigma^3\partial_x u\vert$};
    \node (F) at (-6,-3) {$\Vert\theta\Vert_\infty\displaystyle{\int_0^1}\rho\cv\theta^3$};
    \node (G) at (0,-3) {$\left\vert\displaystyle{\int_0^1}\brokenvert\sigma\brokenvert^3\partial_x u\right\vert$};
    \node (H) at (-2,-6) {$\eta\displaystyle{\int_0^1}\dfrac{(\partial_x\sigma)^2\vert\sigma\vert}{\rho}$};
    \node (I) at (0,-6) {$\displaystyle{\int_0^1}\frac{\vert\sigma\vert^3}{\mu}$};
    \node (J) at (0,3) {$\int_0^1\vert\sigma\partial_x u\vert\theta^2$};
        \draw[->, line width=1pt]
        (K) -- (A)
        node[midway, above] {(g)};
        \draw[->, line width=1pt]
        (K) -- (B)
        node[midway, above] {(g)};

    \draw[->, line width=1pt] 
        (B.south) -- (C.north) 
        node[midway, right] {(e)};
       
    \draw[->, line width=1pt] 
        (A.south) -- (E.north) 
        node[midway, right] {(d)};
        
        \draw[->, line width=1pt] 
        (G.south) -- (H.north) 
        node[midway, above] {(a)};
        \draw[->, line width=1pt] 
        (G.south) -- (I.north) 
        node[midway, right] {(a)};
       
         \draw[->, line width=1pt] 
        (C.west) -- (G.east) 
        node[midway, above, sloped] {(b)};
    
        \draw[->, >=stealth, line width=1pt] 
        (A) -- (C) 
        node[midway, above, yshift=2pt] {(d)};
    \draw[->, line width=1pt] (D)--(C) node[midway, left] {(c)};
    \draw[->, line width=1pt] (D)--(E) node[midway, above, yshift = 2pt] {(c)};
     \draw[->, line width=1pt] (C)--(E) node[midway, above, yshift = 2pt] {(b)};
     \draw[->, line width=1pt] (B)--(F) node[midway, above, yshift = 2pt] {(e)};
          \draw[->, line width=1pt] (J)--(A) node[midway, above, yshift = 2pt] {(f)};
          \draw[->, line width=1pt] (J)--(F) node[midway, above, yshift = 2pt] {(f)};

\end{tikzpicture}
\end{figure}
\subsection{Closure of the estimates \texorpdfstring{\eqref{eq:est1}--\eqref{eq:est3}}{} }
In order to prove (a), we need the
\begin{lemma}\label{GN} For all $\eta>0$, there exists some $C_\eta= C_\eta(\eta, \overline{\rho},{\cal E})>0$ such that, for all $f\in H^1(\T)$,
\begin{equation}
\int_0^1 \vert u f \partial_x f\vert \leq \eta\int_0^1 \frac{(\partial_x f)^2}{\rho} + C_\eta\int_0^1 f^2.
\end{equation}
\end{lemma}
\begin{proof} Using the Hölder's inequality we obtain from \eqref{eq:29}
\begin{equation}\label{eq:plic}
\int_0^1 \vert u f \partial_x f\vert \leq \Vert f\Vert_\infty \int_0^1 \sqrt{\rho} \vert u \vert \frac{\vert \partial_x f\vert }{\sqrt{\rho}}\leq \Vert f\Vert_\infty \sqrt{2\cE}\left(\int_0^1\frac{(\partial_x f)^2}{\rho}\right)^{1/2}.
\end{equation}

Then using the Gagliardo--Nirenberg inequality
\begin{equation}\label{eq:GN}
\Vert f\Vert_\infty^2 \leq \Vert f\Vert_2^2 + 2\Vert f\Vert_2\Vert \partial_x f\Vert_2,
\end{equation}
we get
\begin{align}\label{eq:ploc}
\Vert f\Vert_\infty &\leq \Vert f\Vert_{2} + 2\Vert f\Vert_{2}^{1/2}\Vert \partial_x f\Vert_{2}^{1/2}
\\&\leq \Vert f\Vert_{2} + 2\overline{\rho}^{1/4} \Vert f\Vert_{2}^{1/2} \Vert \partial_x f/\sqrt{\rho}\Vert_{2}^{1/2}.
\end{align}
Thus by Young's inequality, \eqref{eq:plic} and \eqref{eq:ploc} give
\begin{align}
\int_0^1 \vert u f\partial_x f\vert &\leq \sqrt{2\cE}\Vert f\Vert_{2} \Vert \partial_x f/\sqrt{\rho}\Vert_{2} + 2\overline{\rho}^{1/4}\sqrt{2\cE}\Vert f\Vert_{2}^{1/2}\Vert \partial_x f/\sqrt{\rho}\Vert_{2}^{3/2} 
\\&\leq \frac{\cE}{\eta}\int_0^1 f^2 + \frac{\eta}{2}\int_0^1 \frac{(\partial_xf)^2}{\rho} + 54\overline{\rho}\frac{\cE^2}{\eta^3}\int_0^1 f^2 + \frac{\eta}{2}\int_0^1 \frac{(\partial_x f)^2}{\rho}
\\&= C_{\eta} \int_0^1 f^2 + \eta \int_0^1 \frac{(\partial_x f)^2}{\rho}
\end{align}
where
\[C_{\eta}= \frac{\cE}{\eta} + 54\overline{\rho}\frac{\cE^2}{\eta^3}>0. \]
\end{proof}

The following proposition is the rigorous statement associated to Figure \ref{fig:schempreuve}.
\begin{prop}\label{prop:fond}
For all $\eta>0$, there exists some $D_\eta=D_\eta(C_0,\eta,\underline{\rho_0},\overline{\rho_0},\underline{\theta_0},\overline{\theta_0},T)>0$ such that
\begin{align}
\int_0^1 \sigma^4 + \int_0^1\sigma^2(\partial_x u)^2 &+ \int_0^1\sigma^2\theta^2 + \int_0^1\vert\sigma^3\partial_x u\vert + \int_0^1\vert\sigma\partial_x u\vert\theta^2 + \int_0^1\sigma^2\theta\vert\partial_x u\vert\\&\leq D_\eta\Vert\theta\Vert_\infty\int_0^1 \rho\cv\theta^3 + D_\eta(1+\Vert\theta\Vert_\infty)\int_0^1\frac{\vert\sigma\vert^3}{\mu} + \eta \int_0^1 \frac{(\partial_x\sigma)^2\vert\sigma\vert}{\rho}.
\end{align}
\end{prop}

\begin{proof}
It is sufficient to show that each of the arrows of Figure \ref{fig:schempreuve} is verified.
\begin{itemize}

\item[\fbox{(a)}] By integration by parts, we get
\begin{equation}
\int_0^1 \vert \sigma\vert^3\partial_x u = -3\int_0^1 u\sign(\sigma)\sigma^2\partial_x\sigma = -2\int_0^1 u\vert\sigma\vert^{3/2}\partial_x\vert\sigma\vert^{3/2},
\end{equation}
therefore the result applying Lemma \ref{GN} with $f=\vert\sigma\vert^{3/2}$. A similar proof holds replacing $\vert\sigma\vert^3$ by $\sigma^3$.
\item[\fbox{(b)}] From \eqref{eq:defsigma}, we get the inequality
\begin{align}
\int_0^1\sigma^4 &\leq \mu_{\max}\int_0^1\frac{\sigma^4}{\mu} = \mu_{\max}\int_0^1 \sigma^3\partial_x u - \mu_{\max}\int_0^1\sigma^3 \frac{R\rho\theta}{\mu}
\\&\leq \mu_{\max}\left\vert\int_0^1\sigma^3\partial_x u\right\vert + R_{\max}\overline{\rho}\mu_{\max}\Vert\theta\Vert_\infty\int_0^1\frac{\vert\sigma\vert^3}{\mu}.
\end{align}
\item[\fbox{(c)}] From \eqref{eq:defsigma}, we get
\begin{align}
\int_0^1 \vert\sigma^3\partial_x u\vert&\leq \int_0^1 \frac{\sigma^4}{\mu} + \int_0^1\frac{\sigma^3R\rho\theta}{\mu}
\\&\leq \frac{1}{\mu_{\min}}\int_0^1\sigma^4 + R_{\max}\overline{\rho}\Vert\theta\Vert_\infty \int_0^1 \frac{\vert\sigma\vert^3}{\mu}. 
\end{align}
\item[\fbox{(d)}] Using again \eqref{eq:defsigma}, we have
\begin{equation}
(\partial_x u)^2 = \frac{\sigma^2 + 2\sigma p + p^2}{\mu^2}\leq \frac{1}{\mu_{\min}^2}\sigma^2 + \frac{2R_{\max}\overline{\rho}}{\mu_{\min}}\Vert\theta\Vert_\infty\frac{\vert\sigma\vert}{\mu} + \frac{R_{\max}^2\overline{\rho}^2}{\mu_{\min}^2}\theta^2,
\end{equation}
hence
\begin{equation}
\int_0^1\sigma^2(\partial_x u)^2\leq \frac{1}{\mu_{\min}^2}\int_0^1\sigma^4 + \frac{2R_{\max}\overline{\rho}}{\mu_{\min}^2}\Vert\theta\Vert_\infty\int_0^1\frac{\vert\sigma\vert^3}{\mu} + \frac{R_{\max}^2\overline{\rho}^2}{\mu_{\min}^2}\int_0^1\sigma^2\theta^2.
\end{equation}
\item[\fbox{(e)}] By Young's inequality,
\begin{align}
\int_0^1\sigma^2\theta^2&\leq \frac{1}{2}\int_0^1\sigma^4 + \frac{1}{2}\int_0^1\theta^4
\\&\leq\frac{1}{2}\int_0^1\sigma^4 + \frac{1}{2\underline{\rho}\cv_{\min}}\Vert\theta\Vert_\infty\int_0^1\rho\cv\theta^3.
\end{align}
\item[\fbox{(f)}] By Young's inequality,
\begin{align}
\int_0^1\vert\sigma\partial_x u\vert\theta^2&\leq \frac{1}{2}\int_0^1\sigma^2(\partial_x u)^2 + \frac{1}{2}\int_0^1\theta^4
\\&\leq \frac{1}{2}\int_0^1\sigma^2(\partial_x u)^2 + \frac{1}{2\underline{\rho}\cv_{\min}}\Vert\theta\Vert_\infty\int_0^1\rho\cv\theta^3.
\end{align}
\item[\fbox{(g)}] Finally, by Young's inequality again,
\begin{equation}
\int_0^1\sigma^2\theta\vert\partial_x u\vert\leq \frac{1}{2}\int_0^1\sigma^2\theta^2 + \frac{1}{2}\int_0^1 \sigma^2(\partial_x u)^2.
\end{equation}
\end{itemize}
\end{proof}
Proposition \ref{prop:fond} is the main tool in order to close \eqref{eq:est1}--\eqref{eq:est3}.
\begin{prop}
For all $\eta>0$, there exists some $G_{1,\eta}= G_{1,\eta}(D_\eta,\eta, \overline{\rho},T)>0$ such that
\begin{equation}\label{eq:etde1}
\begin{split}
&\frac{d}{dt}\int_0^1 \frac{\vert \sigma\vert^3}{\mu} + \int_0^1 \frac{(\partial_x \sigma)^2\vert\sigma\vert}{\rho}\leq G_{1,\eta}\Vert\theta\Vert_\infty\int_0^1\rho\cv\theta^3 + G_{1,\eta}(1 + \Vert\theta\Vert_\infty)\int_0^1 \frac{\vert\sigma\vert^3}{\mu} + \eta\int_0^1\frac{\vert\partial_x(\kappa\partial_x\theta)\vert^2}{\rho\cv}.
\end{split}
\end{equation}
\end{prop}
\begin{proof}
Starting from \eqref{eq:est2} and by Young's inequality, we get
\begin{align}
\frac{1}{3}\frac{d}{dt}\int_0^1 \frac{\vert \sigma\vert^3}{\mu} &+ 2\int_0^1 \frac{(\partial_x \sigma)^2\vert\sigma\vert}{\rho} = \int_0^1 \frac{1/3-\gamma}{\mu}\vert\sigma\vert^3\partial_x u - \int_0^1\frac{\gamma-1}{\mu}\partial_x(\kappa\partial_x\theta)\sigma\vert\sigma\vert
\\&\leq \frac{\gamma_{\max}-1/3}{\mu_{\min}}\int_0^1 \vert\sigma\vert^3\vert\partial_x u\vert + \frac{\gamma_{\max}-1}{\mu_{\min}}\int_0^1 \vert\partial_x(\kappa\partial_x\theta)\vert \vert\sigma\vert^2
\\&\leq \frac{\gamma_{\max}-1/3}{\mu_{\min}}\int_0^1 \vert\sigma\vert^3\vert\partial_x u\vert + \frac{(\gamma_{\max}-1)^2\overline{\rho}\cv_{\max}}{4\eta\mu_{\min}^2}\int_0^1\sigma^4+ \eta\int_0^1\frac{\vert\partial_x(\kappa\partial_x\theta)\vert^2}{\rho\cv}.
\end{align}
Then we conclude thanks to Proposition \ref{prop:fond}.
\end{proof}

\begin{prop} For all $\eta>0$, there exists some $G_{2,\eta}=G_{2,\eta}(D_\eta,\eta,\overline{\rho},\underline{\rho},T)>0$ such that 
\begin{align}\label{eq:etde3}
\frac{d}{dt}\int_0^1\rho \cv\theta^3 + \int_0^1\kappa (\partial_x\theta)^2\theta &\leq G_{2,\eta}(1 + \Vert\theta\Vert_\infty)\int_0^1\frac{\vert\sigma\vert^3}{\mu} + G_{2,\eta}\Vert\theta\Vert_\infty\int_0^1\cv\rho\theta^3 + \eta \int_0^1\frac{(\partial_x\sigma)^2\vert\sigma\vert}{\rho}.
\end{align}
\end{prop}
\begin{proof}
Multiplying \eqref{eq:thetameso} by $\theta^2$ then integrating by part and using some Young's inequality, we get
\begin{align}
\frac{1}{3}\frac{d}{dt}\int_0^1\rho \cv\theta^3 + 2\int_0^1\kappa (\partial_x\theta)^2\theta &= \int_0^1\sigma(\partial_x u)\theta^2
\\&\leq \frac{1}{\cv_{\min}\underline{\rho}}\int_0^1\sigma^2(\partial_x u)^2 + \frac{\Vert\theta\Vert_\infty}{4}\int_0^1\rho\cv\theta^3.
\end{align}
Hence the result, using Proposition \ref{prop:fond}.
\end{proof}

\begin{prop}\label{prop:17jcrois}
There exists some $G_{3,\eta} = G_{3,\eta}(\eta,\underline{\rho},\overline{\rho},{\cal E})>0$ such that
\begin{align}\label{eq:etde2}
\frac{d}{dt}\int_0^1 \kappa (\partial_x\theta)^2 + \int_0^1 \frac{[\partial_x(\kappa\partial_x\theta)]^2}{\rho\cv} &\leq G_{\eta,3}(1+\Vert\theta\Vert_\infty)\int_0^1\frac{\vert\sigma\vert^3}{\mu} + G_{\eta,3}\Vert\theta\Vert_\infty\int_0^1\cv\rho\theta^3 \\&+ \eta \int_0^1\frac{(\partial_x\sigma)^2\vert\sigma\vert}{\rho} - \int_0^1\kappa(\partial_x\theta)^2\partial_x u.
\end{align}
\end{prop}
\begin{proof}
The result follows immediately using \eqref{eq:est1} then Proposition \ref{prop:fond}.
\end{proof}

To close the estimates, we need to obtain some bounds on
\begin{equation}
\int_0^1\kappa(\partial_x\theta)^2\partial_x u.
\end{equation}
Before proving bounds in general terms, note that in the simple case $\kappa_+ = \kappa_-$, this can be obtained briefly.
\begin{prop}
Assume that $\kappa_+ = \kappa_-$. Then, for all $\eta>0$,
\begin{equation}
\left\vert\int_0^1\kappa(\partial_x\theta)^2\partial_x u\right\vert \leq 2 C_\eta \int_0^1 \kappa\vert\partial_x\theta\vert^2 + \frac{2\eta\cv_{\max}}{\kappa} \int_0^1 \frac{[\partial_x(\kappa\partial_x\theta)]^2}{\rho\cv}.
\end{equation}
\end{prop}
\begin{proof}
Let us observe that, in the case $\kappa_+ = \kappa_-$, we get by integration by parts
\begin{equation}
-\int_0^1 \kappa(\partial_x\theta)^2\partial_x u = 2\kappa\int_0^1 u (\partial_x\theta) \partial_{xx}\theta.
\end{equation}
Using now Lemma \ref{GN} we obtain, for all $\eta>0$,
\begin{align}
-\int_0^1\kappa(\partial_x\theta)^2\partial_x u &\leq 2 C_\eta \int_0^1 \kappa\vert\partial_x\theta\vert^2 + \frac{2\eta}{\kappa} \int_0^1 \frac{[\partial_x(\kappa\partial_x\theta)]^2}{\rho}
\\&\leq 2 C_\eta \int_0^1 \kappa\vert\partial_x\theta\vert^2 + \frac{2\eta\cv_{\max}}{\kappa} \int_0^1 \frac{[\partial_x(\kappa\partial_x\theta)]^2}{\rho\cv} .
\end{align}
\end{proof}
In the general case, we need more integrability informations:
\begin{prop}
For all $\eta>0$, there exists some $G_{4,\eta} = G_{4,\eta}(\eta,D_\eta, \overline{\rho},\underline{\rho})>0$ such that
\begin{equation}
\begin{split}
\frac{d}{dt}\int_0^1 \frac{\sigma^2\theta}{\mu} + \int_0^1 \frac{(\partial_x\sigma)^2\theta}{\rho} &\leq \eta\int_0^1\frac{(\partial_x\sigma)^2\vert\sigma\vert}{\rho} + \eta\int_0^1\frac{[\partial_x(\kappa\partial_x\theta)]^2}{\rho\cv} + G_{4,\eta}\Vert\theta\Vert_\infty\int_0^1\kappa(\partial_x\theta)^2 \\&+ G_{4,\eta}\Vert\theta\Vert_\infty \int_0^1\rho\cv\theta^3 + G_{4,\eta}(1+\Vert\theta\Vert_\infty)\int_0^1 \frac{\vert\sigma\vert^3}{\mu} + \eta\int_0^1 \frac{(\partial_x\sigma)^2\theta}{\rho}.
\end{split}\label{eq:etde4}
\end{equation}
\end{prop}

\begin{proof}
From  and , we get
\begin{align}
\partial_t\left(\frac{\sigma^2\theta}{\mu}\right) &+ \partial_x\left(\frac{\sigma^2\theta}{\mu}u\right) = \left[\partial_t\left(\frac{\sigma^2}{\mu}\right) + \partial_x\left(\frac{\sigma^2}{\mu}u\right)\right]\theta + \frac{\sigma^2 D_t\theta}{\mu} = I_1 + I_2.
\end{align}
First, from \eqref{eq:sigma} and \eqref{eq:thetameso}  we get
\begin{equation}
I_1 = 2\sigma\theta\partial_x\left(\frac{\partial_x\sigma}{\rho}\right) + \frac{(1-2\gamma)}{\mu}\sigma^2\theta\partial_x u - \frac{2(\gamma-1)}{\mu}\sigma\theta\partial_x(\kappa\partial_x\theta),
\end{equation}
\begin{equation}
I_2 =  \frac{\sigma^3\partial_x u}{\mu\rho\cv} + \frac{\sigma^2\partial_x(\kappa\partial_x\theta)}{\mu\rho\cv}.
\end{equation}
Hence, integrating over the torus, we obtain by integration by parts and by Young's inequality
\begin{align}
\int_0^1 I_1 &=  - 2\int_0^1 \frac{(\partial_x\sigma)^2\theta}{\rho} -2\int_0^1 \sigma\partial_x\theta \frac{\partial_x\sigma}{\rho} + \int_0^1\frac{(1-2\gamma)}{\mu}\sigma^2\theta\partial_x u - \int_0^1 \frac{2(\gamma-1)}{\mu}\sigma\theta\partial_x(\kappa\partial_x\theta)
\\&\leq -2\int_0^1\frac{(\partial_x\sigma)^2\theta}{\rho} + \frac{1}{\eta\overline{\rho}\kappa_{\min}}\int_0^1\vert\sigma\vert\kappa(\partial_x\theta)^2 + \eta\int_0^1\frac{(\partial_x\sigma)^2\vert\sigma\vert}{\rho} + \frac{2\gamma_{\max}-1}{\mu_{\min}}\int_0^1\sigma^2\theta\vert\partial_x u\vert \\&+ \frac{2(\gamma_{\max}-1)^2\overline{\rho}\cv_{\max}}{\eta\mu_{\min}^2}\int_0^1\sigma^2\theta^2 +\frac{\eta}{2}\int_0^1 \frac{[\partial_x(\kappa\partial_x\theta)]^2}{\rho\cv},
\end{align}
and
\begin{align}
\int_0^1 I_2 \leq \frac{1}{\mu_{\min}\underline{\rho}\cv_{\min}}\int_0^1 \vert\sigma^3\partial_x u\vert + \frac{1}{2\mu_{\min}^2\underline{\rho}\cv_{\min}\eta}\int_0^1\sigma^4 + \frac{\eta}{2}\int_0^1\frac{[\partial_x(\kappa\partial_x\theta)]^2}{\rho\cv}.
\end{align}
Moreover, by integration by parts
\begin{align}
&\int_0^1\vert\sigma\vert\kappa(\partial_x\theta)^2 = - \int_0^1 \vert\sigma\vert\theta \partial_x(\kappa\partial_x\theta) - \int_0^1 \kappa(\partial_x \vert\sigma\vert)\theta\partial_x \theta\label{eq:lastipp}
\\&\leq \frac{\overline{\rho}\cv_{\max}}{2\eta^2}\int_0^1 \sigma^2\theta^2 + \frac{\eta^2}{2}\int_0^1 \frac{[\partial_x(\kappa\partial_x\theta)]^2}{\rho\cv}  + \frac{\kappa_{\max}\overline{\rho}^2}{4\eta^2\underline{\rho}\kappa_{\min}}\Vert\theta\Vert_\infty\int_0^1 \kappa(\partial_x\theta)^2 + \eta^2\underline{\rho}\kappa_{\min}\int_0^1 \frac{(\partial_x\sigma)^2\theta}{\rho}.
\end{align}
We then prove \eqref{eq:etde4} using Proposition \ref{prop:fond}.
\end{proof}
\begin{prop}
For all $\eta>0$, there exists some $G_{5,\eta} = G_{5,\eta}(\eta,D_\eta,\overline{\rho},\underline{\rho})>0$ such that
\begin{equation}
\begin{split}
\left\vert \int_0^1\kappa(\partial_x\theta)^2\partial_x u\right\vert &\leq G_{5,\eta}\Vert\theta\Vert_\infty\int_0^1\rho\cv\theta^3 + G_{5,\eta}(1+\Vert\theta\Vert_\infty)\int_0^1\frac{\vert\sigma\vert^3}{\mu}+ G_{5,\eta}\Vert\theta\Vert_\infty\int_0^1 \kappa(\partial_x\theta)^2\\&+\eta\int_0^1\frac{(\partial_x\sigma)^2\vert\sigma\vert}{\rho}+\eta \int_0^1 \frac{[\partial_x(\kappa\partial_x\theta)]^2}{\rho\cv}  + \eta \int_0^1 \frac{(\partial_x\sigma)^2\theta}{\rho}.
\end{split}\label{eq:chabada}
\end{equation}
\end{prop}
\begin{proof}
\begin{align}
\left\vert\int_0^1\kappa(\partial_x\theta)^2\partial_x u\right\vert &\leq \int_0^1 \kappa(\partial_x\theta)^2\vert\partial_x u\vert
\\&\leq \frac{1}{\mu_{\min}}\int_0^1\kappa(\partial_x\theta)^2\vert\sigma\vert + \frac{R\overline{\rho}\Vert\theta\Vert_\infty}{\mu_{\min}}\int_0^1\kappa(\partial_x\theta)^2.
\end{align}
As a consequence, using \eqref{eq:lastipp},
\begin{equation}
\begin{split}
\left\vert\int_0^1\kappa(\partial_x\theta)^2\partial_x u\right\vert &\leq \frac{\overline{\rho}\cv_{\max}}{2\eta\mu_{\min}}\int_0^1\sigma^2\theta^2 + \frac{\eta}{2\mu_{\min}}\int_0^1\frac{[\partial_x(\kappa\partial_x\theta)]^2}{\rho\cv} + \left(\frac{\kappa_{\max}\overline{\rho}}{4\eta\mu_{\min}\underline{\rho}\kappa_{\min}} + \frac{R}{\mu_{\min}}\right)\overline{\rho}\Vert\theta\Vert_\infty \int_0^1\kappa(\partial_x\theta)^2 \\&+ \frac{\eta\overline{\rho}\kappa_{\min}}{\mu_{\min}}\int_0^1 \frac{(\partial_x\sigma)^2\theta}{\rho}.
\end{split}
\end{equation}
We then obtain \eqref{eq:chabada} using Proposition \ref{prop:fond}.
\end{proof}
Finally, combining all the previous results, we get
\begin{prop}\label{prop:oncolletout}
There exists some $H_5 = H_5(\overline{\rho_0},\underline{\rho_0},\overline{\theta_0},\underline{\theta_0}, C_0)>0$ such that
\begin{equation}
\begin{split}
&\sup_{0\leq t\leq T}\int_0^1\vert\sigma\vert^3 + \sup_{0\leq t\leq T}\int_0^1\rho\theta^3 + \sup_{0\leq t\leq T}\int_0^1 (\partial_x\theta)^2 + \sup_{0\leq t\leq T}\sigma^2\theta \\&+ \int_0^T\int_0^1 (\partial_x\sigma)^2\vert\sigma\vert + \int_0^T\int_0^1(\partial_x\theta)^2\theta + \int_0^T\int_0^1 [\partial_x(\kappa\partial_x\theta)]^2 + \int_0^T\int_0^1(\partial_x\sigma)^2\theta \leq H_5.
\end{split}
\end{equation}
\end{prop}
\begin{proof}
Adding \eqref{eq:etde1}, \eqref{eq:etde3}, \eqref{eq:etde2}, \eqref{eq:etde4} and using \eqref{eq:chabada}, we get, denoting $G_\eta = \max_{i=1\cdots 5} G_{i,\eta}$,
\begin{equation}
\begin{split}
&\frac{d}{dt}\int_0^1 \frac{\vert\sigma\vert^3}{\mu} + \frac{d}{dt}\int_0^1\rho\cv\theta^3 + \frac{d}{dt}\int_0^1\kappa(\partial_x\theta)^2 + \frac{d}{dt}\int_0^1\frac{\sigma^2\theta}{\mu} \\&+ \int_0^1\frac{(\partial_x\sigma)^2\vert\sigma\vert}{\rho} + \int_0^1\kappa(\partial_x\theta)^2\theta + \int_0^1 \frac{[\partial_x(\kappa\partial_x\theta)]^2}{\rho\cv} + \int_0^1\frac{(\partial_x\sigma)^2\theta}{\rho} \\&\leq 5G_\eta(1+\Vert\theta\Vert_\infty)\int_0^1\frac{\vert\sigma\vert^3}{\mu} + 5G_\eta\Vert\theta\Vert_\infty\int_0^1\rho\cv\theta^3+ 2G_\eta\Vert\theta\Vert_\infty\int_0^1\kappa(\partial_x\theta)^2 \\&+ 3\eta\int_0^1 \frac{[\partial_x(\kappa\partial_x\theta)]^2}{\rho\cv} + 5\eta\int_0^1 \frac{(\partial_x \sigma)^2\vert\sigma\vert}{\rho} + 2\eta\int_0^1\frac{(\partial_x\sigma)^2\theta}{\rho}.
\end{split}
\end{equation}
Choosing $\eta \leq 1/10$, we get
\begin{equation}
\begin{split}
&\frac{d}{dt}\int_0^1 \frac{\vert\sigma\vert^3}{\mu} + \frac{d}{dt}\int_0^1\rho\cv\theta^3 + \frac{d}{dt}\int_0^1\kappa(\partial_x\theta)^2 + \frac{d}{dt}\int_0^1\frac{\sigma^2\theta}{\mu} \\&+ \frac{1}{2}\int_0^1\frac{(\partial_x\sigma)^2\vert\sigma\vert}{\rho} + \frac{1}{2}\int_0^1\kappa(\partial_x\theta)^2\theta + \frac{1}{2}\int_0^1 \frac{[\partial_x(\kappa\partial_x\theta)]^2}{\rho\cv} + \frac{1}{2}\int_0^1\frac{(\partial_x\sigma)\theta}{\rho} \\&\leq 5G_\eta(1+\Vert\theta\Vert_\infty)\left(\int_0^1\frac{\vert\sigma\vert^3}{\mu}+\int_0^1\rho\cv\theta^3 + \int_0^1\kappa(\partial_x\theta)^2\right).
\end{split}
\end{equation}
Hence, by the Grönwall's lemma,
\begin{equation}
\begin{split}
&\int_0^1 \frac{\vert\sigma\vert^3}{\mu} + \int_0^1\rho\cv\theta^3 + \int_0^1\kappa(\partial_x\theta)^2 + \int_0^1\frac{\sigma^2\theta}{\mu} \\&+ \frac{1}{2}\int_0^T\int_0^1\frac{(\partial_x\sigma)^2\vert\sigma\vert}{\rho} + \frac{1}{2}\int_0^T\int_0^1\kappa(\partial_x\theta)^2\theta + \frac{1}{2}\int_0^T\int_0^1 \frac{[\partial_x(\kappa\partial_x\theta)]^2}{\rho\cv} + \frac{1}{2}\int_0^T\int_0^1\frac{(\partial_x\sigma^2)\theta}{\rho}
\\&\leq \exp(5G_\eta(T+H_4))C_0'
\end{split}
\end{equation}
where
\begin{equation}
C_0' = 3{\mu_{\max}^2}C_0 + 3\frac{R_{\max}^3\overline{\rho_0}^3\overline{\theta_0}^3}{\mu_{\min}} + \cv_{\max}\overline{\rho_0}\overline{\theta_0}^3 + \kappa_{\max}C_0 + 2\frac{R_{\max}^2\overline{\rho_0}^2\overline{\theta_0}^3}{\mu_{\min}} + \frac{2}{3}\mu_{\max}C_0 + \frac{2}{3}\mu_{\max}\overline{\theta_0}^3.
\end{equation}
\end{proof}
\begin{prop}\label{prop:21}
There exists some $H_6 = H_6(\overline{\rho_0},\underline{\rho_0},\overline{\theta_0}, \underline{\theta_0}, C_0)>0$ such that
\begin{equation}
\sup_{0\leq t\leq T}\int_0^1\sigma^2 + \int_0^T\int_0^1 (\partial_x\sigma)^2\leq H_6.
\end{equation}
\end{prop}
\begin{proof}
Let us apply \eqref{eq:ka} with $k=2$. We get
\begin{align}
\frac{1}{2}\frac{d}{dt}\int_0^1 \frac{\sigma^2}{\mu} + \int_0^1 \frac{(\partial_x\sigma)^2}{\rho} &= \int_0^1 \frac{1-2\gamma}{2 \mu}\sigma^2\partial_x u - \int_0^1 \frac{\gamma-1}{\mu}[\partial_x(\kappa\partial_x\theta)]\sigma
\\&\leq \frac{2\gamma_{\max}-1}{2\mu_{\min}}\int_0^1\sigma^2\vert\partial_x u\vert + \frac{\gamma_{\max}-1}{2\mu_{\min}}\int_0^1 \vert\partial_x(\kappa\partial_x\theta)\vert^2 + \frac{\gamma_{\max}-1}{2}\int_0^1\frac{\sigma^2}{\mu}.
\end{align} 
And
\begin{equation}
\vert\sigma^2\partial_x u\vert = \left\vert\frac{\sigma^3}{\mu} + \frac{\sigma^2R\rho\theta}{\mu}\right\vert \leq \frac{\vert\sigma\vert^3}{\mu_{\min}} + R_{\max}\overline{\rho}\Vert\theta\Vert_\infty \frac{\sigma^2}{\mu},
\end{equation}
thus
\begin{equation}
\begin{split}
\frac{1}{2}\frac{d}{dt}\int_0^1\frac{\sigma^2}{\mu} + \int_0^1\frac{(\partial_x\sigma)^2}{\rho} &\leq \frac{2\gamma_{\max}-1}{2\mu_{\min}^2}\int_0^1\vert\sigma\vert^3 + \frac{(2\gamma_{\max}-1)R_{\max}\overline{\rho}}{2\mu_{\min}}\Vert\theta\Vert_\infty\int_0^1 \frac{\sigma^2}{\mu}
\\& +\frac{\gamma_{\max}-1}{2\mu_{\min}}\int_0^1 \vert\partial_x(\kappa\partial_x\theta)\vert^2 + \frac{\gamma_{\max}-1}{2}\int_0^1 \frac{\sigma^2}{\mu}.
\end{split}
\end{equation}
By Grönwall's lemma, we get
\begin{equation}
\begin{split}
&\frac{1}{2}\sup_{0\leq t\leq T} \int_0^1\frac{\sigma^2}{\mu}+ \int_0^T\int_0^1\frac{(\partial_x\sigma)^2}{\rho} \\&\leq \left(\frac{1}{2\mu_{\min}}\int_0^1{\sigma_0^2} + \left(\frac{2\gamma_{\max}-1}{2\mu_{\min}^2}+\frac{\gamma_{\max}-1}{2\mu_{\min}}\right)H_5 \right)\exp\left(\frac{\gamma_{\max}-1}{2}T + \frac{(2\gamma_{\max}-1)R_{\max}\overline{\rho}}{2\mu_{\min}}H_7 \right)
\end{split},
\end{equation}
then Proposition \ref{prop:21} because $\rho\geq \underline{\rho}$ and $\mu\geq \mu_{\min}$.
\end{proof}
\begin{prop}\label{prop:sigmainf}
The following inequality holds:
\begin{equation}
\int_0^T\Vert\sigma\Vert_\infty^3 \leq 3(T+1)H_5.
\end{equation}
\end{prop}
\begin{proof}
Using the Gagliardo Nirenberg's inequality \eqref{eq:GN}, we get by Young's inequality,
\begin{align}
\int_0^T\Vert\sigma\Vert_\infty^3 = \int_0^1\Vert\vert\sigma\vert^{3/2}\Vert_\infty^2 &\leq 2\int_0^T\int_0^1 \vert\sigma\vert^3 + \frac{9}{4}\int_0^T\int_0^1(\partial_x\sigma)^2\vert\sigma\vert
\\&\leq 3(T + 1)H_5.
\end{align}
\end{proof}
\begin{prop}
There exists some $\overline{\theta}=\overline{\theta}(\overline{\rho_0},\underline{\rho_0},\overline{\theta_0},\underline{\theta_0},C_0)>0$ such that
\begin{equation}
\text{for almost all }(t,x)\in [0,T]\times\T,\quad \theta(t,x) \leq \overline{\theta}.
\end{equation}
\end{prop}
\begin{proof}
Starting with \eqref{eq:42} with $\beta :R\rightarrow \R$ some convex and non decreasing function, then integrating over the torus, we get
\begin{equation}
\frac{d}{dt}\int_0^1\rho\cv\beta(\theta) + \int_0^1 \kappa(\partial_x\theta)^2\beta''(\theta) = \int_0^1(\sigma\partial_x u)\beta'(\theta) = \int_0^1\frac{\sigma^2}{\mu}\beta'(\theta) + \int_0^1 \frac{R\sigma}{\mu}\rho\theta\beta'(\theta)
\end{equation} 
therefore
\begin{equation}
\frac{d}{dt}\int_0^1\rho\cv\beta(\theta) \leq \frac{\Vert\sigma\Vert_\infty^2}{\mu_{\min}}\int_0^1\beta'(\theta) + \frac{R_{\max}}{\mu_{\min}}\Vert\sigma\Vert_\infty\int_0^1\rho\theta\beta'(\theta).
\end{equation}
since $\beta''(\theta)\geq 0$ and $\beta'(\theta)\geq 0$. Choosing $\beta : x\mapsto x^k$ for $k\in \N^*$, we get
\begin{equation}
\frac{d}{dt}\int_0^1\rho\cv\theta^k \leq \frac{k\Vert\sigma\Vert_\infty^2}{\mu_{\min}}\int_0^1\theta^{k-1} + \frac{R_{\max}k}{\mu_{\min}\cv_{\min}}\Vert\sigma\Vert_\infty \int_0^1\rho\cv\theta^k.
\end{equation}
Multiplying this last equation by 
\begin{equation}
\frac{1}{k}\left(\int_0^1\rho\cv\theta^k\right)^{1/k-1},
\end{equation}
we obtain
\begin{equation}\label{eq:peche}
\frac{d}{dt}\left(\int_0^1\rho\cv\theta^k\right)^{1/k} \leq \frac{\Vert\sigma\Vert_\infty^2}{\mu_{\min}}\int_0^1\theta^{k-1}\left(\int_0^1\rho\cv\theta^k\right)^{1/k-1} +\frac{R_{\max}}{\mu_{\min}\cv_{\min}}\Vert\sigma\Vert_\infty\left(\int_0^1\rho\cv\theta^k\right)^{1/k}.
\end{equation}
Moreover, by Jensen's inequality,
\begin{equation}
\int_0^1\theta^{k-1}\leq \left(\int_0^1\theta^k\right)^{1-1/k},
\end{equation}
hence
\begin{equation}\label{eq:poire}
\int_0^1\theta^{k-1}\left(\int_0^1\rho\cv\theta^k\right)^{1/k-1}\leq \overline{\rho}^{1/k-1}\cv_{\max}^{1/k-1}.
\end{equation}
Finally, putting \eqref{eq:poire} in \eqref{eq:peche} then taking $k\rightarrow +\infty$, we get
\begin{equation}
\frac{d}{dt}\Vert\theta\Vert_\infty\leq \frac{\Vert\sigma\Vert_\infty^2}{\overline{\rho}\mu_{\min}\cv_{\max}} + \frac{R_{\max}}{\mu_{\min}\cv_{\min}}\Vert\sigma\Vert_\infty\Vert\theta\Vert_\infty.
\end{equation}
Finally, by Grönwall's lemma,
\begin{equation}
\sup_{0\leq t\leq T}\Vert\theta\Vert_\infty(t) \leq \left(\overline{\theta_0} + \frac{(2T+1)H_5}{\overline{\rho}\mu_{\min}\cv_{\max}}\right)\exp\left(\frac{R_{\max}}{\mu_{\min}\cv_{\min}}\sqrt{2T+1}\sqrt{T}\right) =: \overline{\theta}.
\end{equation}
\end{proof}
\begin{prop}
There exists some $H_7=H_7(C_0,\overline{\rho_0},\underline{\rho_0},\overline{\theta_0},\underline{\theta_0},T)>0$ such that
\begin{equation}\label{eq:domusss}
\sup_{0\leq t\leq T}\int_0^1\vert\partial_x u\vert^3 + \int_0^T\int_0^1\vert\partial_t u\vert^2  + \int_0^T\int_0^1\vert\partial_t\theta\vert^2+ \int_0^T\Vert\partial_x u\Vert_\infty^3 + \sup_{[0,T]\times\T}\vert u\vert^2\leq H_7
\end{equation}
\end{prop}
\begin{proof}
We get
\begin{equation}
\vert\partial_x u\vert^3 \leq 3\frac{\sigma^3}{\mu^3} + 3\frac{R^3\rho^3\theta^3}{\mu^3},
\end{equation}
thus, from Proposition \ref{prop:oncolletout} and Proposition \ref{prop:sigmainf},
\begin{equation}\label{eq:cestqueledbut}
\int_0^1\vert\partial_x u\vert^3 \leq \frac{3H_5}{\mu_{\min}^3} + \frac{3R_{\max}^3\overline{\rho}^3\overline{\theta}^3}{\mu_{\min}^3},\quad \int_0^T\Vert\partial_x u\Vert_\infty^3 \leq \frac{9(T+1)H_5}{\mu_{\min}^3} + \frac{3TR_{\max}^3\overline{\rho}^3\overline{\theta}^3}{\mu_{\min}^3}.
\end{equation} 
Then, by the Gagliardo--Nirenberg inequality \eqref{eq:GN},
\begin{equation}
\sup_{[0,T]\times\T}\vert u\vert^2\leq 2\int_0^1\vert u\vert^2 + \int_0^1\vert\partial_x u\vert^2 \leq \frac{4{\cal E}}{\underline{\rho}} + \left(\int_0^1\vert\partial_x u\vert^3\right)^{2/3}.
\end{equation}
Moreover,
\begin{align}
\int_0^T \int_0^1\vert\partial_t u\vert^2 &\leq 2\int_0^T\int_0^1 \vert u\partial_x u\vert^2 + 2\int_0^T\int_0^1\left\vert\frac{\partial_x\sigma}{\rho}\right\vert^2\nonumber
\\&\leq 2 T \left(\sup_{[0,T]\times\T} \vert u\vert^2\right) \left(\int_0^1 \vert\partial_x u\vert^3\right)^{2/3} + \frac{2H_6}{\underline{\rho}^2}.
\end{align}
Finally,
\begin{align}
\int_0^T\int_0^1\vert\partial_t \theta\vert^2 &\leq 2\int_0^T\int_0^1 \left\vert\frac{u\partial_x\theta}{\rho\cv}\right\vert^2 + 2 \int_0^T\int_0^1 \left\vert \frac{\sigma\partial_x u}{\rho\cv}\right\vert^2 + 2\int_0^T\int_0^1 \left\vert\frac{\partial_x(\kappa\partial_x\theta)}{\rho\cv}\right\vert^2\nonumber
\\&\leq \frac{2TH_5}{\underline{\rho}^2\cv_{\min}^2}\sup_{[0,T]\times\T} \vert u\vert^2 + \frac{2}{\underline{\rho}^2\cv_{\min}^2}\int_0^T(1+\Vert\sigma\Vert_\infty^3)\left(\int_0^1\vert\partial_x u\vert^3\right)^{2/3} + \frac{2H_5}{\underline{\rho}^2\cv_{\min}^2}.\label{eq:daccorddaccord}
\end{align}
Combining \eqref{eq:cestqueledbut}--\eqref{eq:daccorddaccord} gives \eqref{eq:domusss}.
\end{proof}
\section{Homogenization}\label{sec:homog}
We present here the different steps required to perform the homogenization, in order to prove Theorem \ref{thm:main}. Some arguments are only sketched because it is very similar to that in \cite{BrBuGJLa}. First, we establish the strong convergence of velocity and temperature. Then, the Young measure $\nu$ is introduced in order to describe the weak convergence of the density and the color function. Roughly speaking, the knowledge of this Young measure allows us to determine the limit of all nonlinear terms in $c$ and $\rho$, and thus concludes the proof of Theorem \ref{thm:main}. The last step is to characterize the Young measure. Initially considering this measure as a convex combination of two Dirac measures, denoted by $\nu_0$, weighted by the volume fraction (which is consistent with the modeling presented in the introduction), we want to show that this ansatz remains true over time. The strategy is as follows:
\begin{itemize}
\item Assuming that the Young measure $\nu$ satisfies the ansatz, we formally determine the macroscopic system \eqref{eq:cmacrotemp}--\eqref{eq:energmacrotemp}.
\item Conversely, by constructing solutions to a macroscopic problem close to \eqref{eq:cmacrotemp}--\eqref{eq:energmacrotemp} (but much easier because $u$, $\theta$ and $\sigma$ are seen as given), we construct a candidate Young measure $\widetilde{\nu}$, which coincides with the Young measure $\nu$ at time $t=0$.
\item We show that both $\nu$ and $\widetilde{\nu}$ satisfy the kinetic equation \eqref{eq:kineq}.
\item Finally, we prove that the kinetic equation \eqref{eq:kineq} has a unique solution equal to $\nu_0$ at time $t=0$.
\end{itemize}

In this section, Let us fix $C_0,\underline{\rho_0},\overline{\rho_0},\underline{\theta_0},\overline{\theta_0}>0$. For all $\varepsilon>0$, we consider some $(c_0^\varepsilon,\rho_0^\varepsilon,u_0^\varepsilon,\theta_0^\varepsilon)\in L^2(\T)^4$ verifying \eqref{eq:debinitcond}--\eqref{eq:fininitcond} with $C_0,\underline{\rho_0},\overline{\rho_0},\underline{\theta_0},\overline{\theta_0}>0$. We denote by $(c^\varepsilon,\rho^\varepsilon,u^\varepsilon,\theta^\varepsilon)$ the "à la Hoff" solution of \eqref{eq:cmeso}--\eqref{eq:energmeso} with initial condition $(c_0^\varepsilon,\rho_0^\varepsilon,u_0^\varepsilon,\theta_0^\varepsilon)$.
\subsection{Strong convergence on \texorpdfstring{$u^\varepsilon$ and $\theta^\varepsilon$}{ueps and thetaeps}}

\begin{prop}\label{prop:twin}
There exists some $u\in L^\infty(0,T,W^{1,3}(\T))\cap L^3(0,T,W^{1,\infty}(\T))$ and \newline$\theta\in L^\infty(0,T,W^{1,2}(\T))\cap L^2(0,T,W^{1,\infty}(\T))$ such that, up to some subsequence,
\begin{equation}\label{eq:giletmoche}
u^\varepsilon \tendf{\varepsilon}{0} u\quad\text{in }L^2(0,T,H^1(\T)),
\end{equation}
\begin{equation}\label{eq:papille}
u^\varepsilon \tend{\varepsilon}{0} u\quad \text{in } L^2(0,T,L^2(\T)),
\end{equation}
\begin{equation}\label{eq:gilettresmoche}
\theta^\varepsilon\tendf{\varepsilon}{0} \theta\quad\text{in }L^2(0,T,H^1(\T)),
\end{equation}
\begin{equation}\label{eq:narine}
\theta^\varepsilon\tend{\varepsilon}{0}\theta\quad\text{in }L^2(0,T,L^2(\T)).
\end{equation}
\end{prop}
\begin{proof}
By Theorem \ref{thm:bounds}, $(u^\varepsilon)$ and $(\theta^\varepsilon)$ are bounded in $H^1(0,T, H^1(\T))$, uniformly in $\varepsilon$. Then we obtain \eqref{eq:papille}, \eqref{eq:narine} by the Rellich–Kondrachov theorem. The unknowns $u$ and $\theta$ have the announced regularity passing to the limit in \eqref{eq:def14}.
\end{proof}
\begin{remark}
A legitimate question that one might ask is whether $(\sigma^\varepsilon)$ converges strongly in $L^2(0,T,L^2(\T))$. As seen in \cite{BrBuGJLa}, to answer positively, it would surely be necessary to show an estimate on the so-called “Hoff's second energy.” We do not discuss this in this article, as strong convergence in $L^2$ is not useful in the rest of the analysis.
\end{remark}
\subsection{Describing the macroscopic model with Young measures}
The quantities $c^\varepsilon$ and $\rho^\varepsilon$ are oscillatory more and more strongly when $\varepsilon\rightarrow 0$. Thus, passing to the limit in nonlinear functions of $c^\varepsilon, \rho^\varepsilon$ presents some difficulty. Roughly speaking, we will have for example
\begin{equation}
\lim_{\varepsilon\rightarrow 0} c^\varepsilon\rho^\varepsilon \neq (\lim_{\varepsilon\rightarrow 0} c^\varepsilon)(\lim_{\varepsilon\rightarrow 0} \rho^\varepsilon)\quad a.e.
\end{equation}  
In order to understand the limit of nonlinear functions of $c^\varepsilon$ and $\rho^\varepsilon$, let us introduce the Young measure $\nu^\varepsilon : [0,T]\times\T \rightarrow {\cal P}(K)$ defined as
\begin{equation}
\text{for all }(t,x)\in [0,T]\times\T,\quad \text{for all }\beta\in C(K),\quad  \langle\nu_{t,x}^\varepsilon,\beta\rangle = \beta(c^\varepsilon(t,x),\rho^\varepsilon(t,x)),
\end{equation}
where $K = [0,1]\times [\underline{\rho},\overline{\rho}]$.
The measure $\nu^\varepsilon$ converges in the following sense:
\begin{prop}
There exists some $\nu\in C(0,T,{\cal M}(K\times\T))\cap L^1(]0,T[\times\T), C(K))'$ such that, up to an extraction
\begin{equation}
\nu^\varepsilon \tend{\varepsilon}{0} \nu \quad \text{in }C(0,T,{\cal M}(K\times\T))\cap L^1(]0,T[\times\T), C(K))'.
\end{equation}
\end{prop}
The proof of this proposition is a straightforward adaptation of the proof of Lemma 5.1. in \cite{BrBuGJLa}. Note that $\nu$ encodes the limit of the nonlinear functions of $c^\varepsilon$ and $\rho^\varepsilon$. Indeed, observe that if $\beta\in C(K)$, then
\begin{equation}
\beta(c^\varepsilon,\rho^\varepsilon) = \langle\nu^\varepsilon,\beta\rangle \tend{\varepsilon}{0} \langle \nu,\beta\rangle \quad \text{in }L^\infty([0,T]\times \T).
\end{equation}
\begin{prop}\label{prop:lala}
There exists some $\sigma\in L^2(0,T,H^1(\T))\cap L^\infty(0,T,L^3(\T))$ and $\mathfrak{f}\in L^2(0,T,H^1(\T))\in L^\infty(0,T,L^2(\T))$ such that, up to a subsequence, 
\begin{equation}\label{eq:chine}
\sigma^\varepsilon \tendf{\varepsilon}{0} \sigma\quad \text{in }L^2(0,T,H^{1}(\T)),
\end{equation}
\begin{equation}\label{eq:inde}
\kappa(c^\varepsilon)\partial_x\theta^\varepsilon \tendf{\varepsilon}{0} \mathfrak{f} \quad\text{in }L^2(0,T,H^{1}(\T)),
\end{equation}
\begin{equation}\label{eq:france}
\text{for all }\beta\in C(K),\quad \langle \nu^\varepsilon,\beta\rangle \sigma^\varepsilon \tendf{\varepsilon}{0} \langle \nu,\beta\rangle \sigma\quad\text{in }L^2(0,T,L^2(\T)),
\end{equation}
\begin{equation}\label{eq:espagne}
\text{for all }\beta\in C(K),\quad \langle \nu^\varepsilon,\beta\rangle \kappa(c^\varepsilon)\partial_x\theta^\varepsilon \tendf{\varepsilon}{0} \langle \nu,\beta\rangle \mathfrak{f}\quad\text{in }L^2(0,T,L^2(\T)).
\end{equation}
\end{prop}
\begin{proof}
By Theorem \ref{thm:bounds}, $(\sigma^\varepsilon)$ and $(\theta^\varepsilon)$ are bounded in $L^2(0,T,H^{1}(\T)$ uniformly in $\varepsilon$. As a consequence \eqref{eq:chine} and \eqref{eq:inde} by the Banach-Alaoglu's theorem.

Let $\beta\in C^1(K)$. Then $(\beta(c^\varepsilon,\rho^\varepsilon))_\varepsilon$ is bounded in $L^2(0,T,L^2(\T))$ uniformly in $\varepsilon$. Moreover, from \eqref{eq:cmeso} and \eqref{eq:contmeso},
\begin{equation}
\partial_t \beta(c^\varepsilon,\rho^\varepsilon) = \partial_x(u^\varepsilon\beta(c^\varepsilon,\rho^\varepsilon)) + (\beta(c^\varepsilon,\rho^\varepsilon)- \rho^\varepsilon\partial_\rho\beta(c^\varepsilon,\rho^\varepsilon)) \partial_x u^\varepsilon.
\end{equation}
Thus, using Theorem \ref{thm:bounds}, $\partial_t \beta(c^\varepsilon,\rho^\varepsilon)$ is bounded in $L^2(0,T,H^{-1}(\T))$. As the inclusion $H^{-1}(\T)\subset L^2(\T)$ is compact, we can apply the Aubin-Lions lemma to get
\begin{equation}\label{eq:norita}
\langle \nu^\varepsilon,\beta\rangle  = \beta(c^\varepsilon,\rho^\varepsilon) \tend{\varepsilon}{0} \langle \nu,\beta\rangle\quad\text{in }L^2(0,T,H^{-1}(\T)).
\end{equation}
The convergence \eqref{eq:norita} is still true for $\beta\in C(K)$ by a straightforward density argument. Combining \eqref{eq:chine}, \eqref{eq:inde} and \eqref{eq:norita} gives \eqref{eq:france}, \eqref{eq:espagne}.
\end{proof}
\begin{cor}
The following equalities hold:
\begin{equation}\label{eq:fluxeff}
\sigma = \frac{1}{\langle\nu, 1/\mu\rangle}\partial_x u - \frac{\langle \nu, R\rho/\mu\rangle}{\langle \nu,1/\mu\rangle}\theta,
\end{equation}
\begin{equation}\label{eq:fouriereff}
\mathfrak{f} = \frac{1}{\langle\nu, 1/\kappa \rangle}\partial_x\theta.
\end{equation}
\end{cor}
\begin{proof}
Starting from the equality
\begin{equation}
\partial_x u^\varepsilon = \frac{\sigma^\varepsilon}{\mu^\varepsilon} + \frac{R^\varepsilon\rho^\varepsilon}{\mu^\varepsilon}\theta^\varepsilon,
\end{equation}
we get, passing to the limit $\varepsilon\rightarrow 0$ then using Proposition \ref{prop:lala}, 
\begin{equation}
\partial_x u = \langle \nu, 1/\mu\rangle \sigma + \langle \nu, R\rho/\mu\rangle\theta.
\end{equation}
We then obtain \eqref{eq:fluxeff} dividing by $\langle \nu, 1/\mu\rangle$. Similarly, writing
\begin{equation}
\partial_x\theta^\varepsilon = \frac{1}{\kappa^\varepsilon}\kappa^\varepsilon\partial_x\theta^\varepsilon
\end{equation}
then passing to the limit $\varepsilon\rightarrow 0$ and using Proposition \ref{prop:lala},
\begin{equation}
\partial_x\theta = \langle \nu, 1/\kappa\rangle \mathfrak{f},
\end{equation}
therefore \eqref{eq:fouriereff}.
\end{proof}
At this stage, we can pass to the limit in \eqref{eq:cmeso}--\eqref{eq:energ} by using the Young measures. The macroscopic system that we obtained was already found in \cite{AZ01}, in a Lagrangian point of view.
\begin{prop}\label{thm:33}
Let us denote by $\alpha_+ = \langle \nu, c\rangle$, $\alpha_- = \langle \nu, 1-c\rangle$, and (by abusing the notation) $\rho = \langle \nu,\rho\rangle$. Then $(\alpha_\pm,\rho,u,\theta)$ are solutions of the equations
\begin{empheq}[left=\empheqlbrace]{alignat=1}
\partial_t \alpha_+ + \partial_x (u\alpha_+) &= \left\langle\nu, \frac{c}{\mu}\right\rangle\sigma + \left\langle\nu, \frac{Rc\rho}{\mu}\right\rangle\theta\label{eq:cmacrotemptemp},\\
\partial_t \alpha_- + \partial_x (u\alpha_-) &= \left\langle\nu, \frac{1-c}{\mu}\right\rangle\sigma + \left\langle\nu, \frac{R(1-c)\rho}{\mu}\right\rangle\theta\label{eq:1mcmacrotemptemp},\\
\partial_t \langle \nu,c\rho\rangle + \partial_x (\langle \nu,c\rho\rangle u) &= 0
\label{eq:contcmacrotemptemp},\\
\partial_t \langle \nu,(1-c)\rho\rangle + \partial_x (\langle \nu,(1-c)\rho\rangle u) &= 0
\label{eq:cont1mcmacrotemptemp},\\
\partial_t (\rho u) + \partial_x (\rho u^2) &=\partial_x\sigma
\label{eq:momtsmacrotemptemp},\\
\partial_t (\langle\nu, \cv \rho\rangle\theta) + \partial_x(\langle \nu, \cv\rho\rangle \theta u) &= \sigma\partial_x u +\partial_x \mathfrak{f}
\label{eq:energmacrotempemp}.
\end{empheq}
\end{prop}
\begin{proof}
\eqref{eq:contcmacrotemptemp}--\eqref{eq:momtsmacrotemptemp} using \eqref{eq:contmeso}, \eqref{eq:cprems} and \eqref{eq:momtsmeso}.
Note that
\begin{align}
0 &= \partial_t c^\varepsilon + u^\varepsilon\partial_x c^\varepsilon \\&= \partial_t c^\varepsilon + \partial_x(c^\varepsilon u^\varepsilon) - c^\varepsilon\partial_x u^\varepsilon
\\&=\partial_t c^\varepsilon + \partial_x(c^\varepsilon u^\varepsilon) - \frac{c^\varepsilon}{\mu^\varepsilon}\sigma^\varepsilon - \frac{c^\varepsilon R^\varepsilon\rho^\varepsilon}{\mu^\varepsilon}\theta^\varepsilon.
\end{align}
Hence \eqref{eq:cmacrotemptemp} passing to the limit by using Proposition \ref{prop:twin} and Proposition \ref{prop:lala}.
Moreover,
\begin{equation}
\sigma^\varepsilon\partial_x u^\varepsilon = \partial_x(\sigma^\varepsilon u^\varepsilon) - u^\varepsilon\partial_x\sigma^\varepsilon \tendf{\varepsilon}{0}\partial_x(\sigma u) - u\partial_x\sigma\quad \text{in } L^1(0,T,L^1(\T)),
\end{equation}
hence \eqref{eq:energmacrotempemp}.
\end{proof}
Let us choose initial conditions of the form
\begin{equation}
\rho_0^\varepsilon = c_0^\varepsilon\rho_{0,+}^\varepsilon + (1-c_0^\varepsilon)\rho_{0,-}^\varepsilon,
\end{equation}
with $(\rho_{0,\pm}^\varepsilon)_{\varepsilon>0}\subset L^\infty(\T)$, and assume that there exists some $\overline{\rho_0}>\underline{\rho_0}>0$ such that
\begin{equation}
\text{for all }\varepsilon>0,\quad \underline{\rho_0} \leq \rho_{0,\pm}^\varepsilon \leq \overline{\rho_0}.
\end{equation} 
Assume moreover that there exists some $\rho_{0,\pm}\in L^\infty(\T)$ such that
\begin{equation}
\rho_{0,\pm}^\varepsilon \tend{\varepsilon}{0}\rho_{0,\pm}\quad \text{in }L^2(\T).
\end{equation}
Observe that, as $c_0^\varepsilon \in \{0,1\}$, we get the existence of some $\alpha_{0,\pm}$ such that
\begin{equation}
c_0^\varepsilon\tend{\varepsilon}{0} \alpha_{0,+}\quad\text{in }L^\infty(\T)-\star,\quad 1-c_0^\varepsilon\tend{\varepsilon}{0}\alpha_{0,-}\quad\text{in }L^\infty(\T)-\star.
\end{equation}
Then, for all $\beta\in C(K)$,
\begin{align}
\beta(c_0^\varepsilon,\rho_0^\varepsilon) = c_0^\varepsilon\beta(1,\rho_{0,+}^\varepsilon) + (1- c_0^\varepsilon)\beta(0,\rho_{0,-}^\varepsilon)\tend{\varepsilon}{0} \alpha_{0,+}\beta(1,\rho_{0,+}) + \alpha_{0,-}\beta(0,\rho_{0,-}) \quad \text{in }L^\infty(\T)-\star.
\end{align}
Defining the Young measure $\nu_0^\varepsilon$ for all $\varepsilon>0$ by
\begin{equation}
\text{for all }\beta\in C(K),\quad \langle \nu_0^\varepsilon,\beta\rangle := c_0^\varepsilon\beta(1,\rho_{0,+}^\varepsilon) + (1-c_0^\varepsilon)\beta(0,\rho_{0,-}^\varepsilon), 
\end{equation}
we then deduce that
\begin{equation}
\nu_0^\varepsilon \rightarrow \alpha_{0,+}\delta_{1,\rho_{0,+}} + \alpha_{0,-}\delta_{0,\rho_{0,-}} \quad \text{in } {\cal M}(K\times\T).
\end{equation}
In particular, by uniqueness of the limit,
\begin{equation}
\text{for almost all }x\in\T,\quad\nu_{0,x} = \alpha_{0,+}(x)\delta_{1,\rho_{0,+}(x)} + \alpha_{0,-}(x)\delta_{0,\rho_{0,-}(x)}.
\end{equation}
\subsection{Construction of a candidate Young measure \texorpdfstring{$\widetilde{\nu}$}{nu tilde}}
The key to obtain the macroscopic system is to show that this structure of the initial measure as some convex combination of two Dirac measures is propagated during the time. Proposition \ref{thm:33} then allows to prove the following
\begin{prop}
Let us assume that there exists some $\alpha_\pm \in L^\infty([0,T]\times\T)$ and $\rho_\pm\in L^\infty([0,T]\times\T)$ such that $\alpha_+ + \alpha_- = 1$ and
\begin{equation}\label{eq:ansatz}
\text{for almost all }(t,x)\in [0,T]\times \T,\quad \nu_{t,x} = \alpha_+(t,x)\delta_{1,\rho_+(t,x)} + \alpha_-(t,x)\delta_{0,\rho_-(t,x)}.
\end{equation} 
Then $(\alpha_\pm,\rho_\pm,u,\theta)$ is a solution to \eqref{eq:cmacrotemp}--\eqref{eq:energmacrotemp}.
\end{prop}
It remains to be shown that the Young measure $\nu$ is indeed of the form \eqref{eq:ansatz}.
Let us construct some candidate Young measure $\widetilde{\nu}$ of the form \eqref{eq:ansatz}. Injecting \eqref{eq:ansatz} in \eqref{eq:cmacrotemptemp} and \eqref{eq:1mcmacrotemptemp} we get

\begin{equation}\label{eq:souris}
\partial_t \alpha_\pm + u\partial_x\alpha_\pm = \frac{\alpha_\pm}{\mu_\pm}\left(\sigma - (\mu_\pm\partial_x u - R_\pm\rho_\pm\theta)\right)
\end{equation}
\begin{equation}\label{eq:chat}
\partial_t(\alpha_\pm\rho_\pm) + \partial_x(\alpha_\pm\rho_\pm u) = 0.
\end{equation}
And the converse is straightforward:
\begin{prop}\label{prop:etonconclut}
Assume that $\nu$ satisfies \eqref{eq:ansatz}, where $(\alpha_\pm,\rho_\pm)$ is a solution of \eqref{eq:souris}--\eqref{eq:chat}. Then $(\alpha_\pm,\rho_\pm,u,\theta)$ is a solution to \eqref{eq:cmacrotemp}--\eqref{eq:energmacrotemp}.
\end{prop}
The following proposition gives the existence of such a solution:
\begin{prop}\label{prop:existence}
Let $u\in L^1(0,T,W^{1,\infty}(\T))$, $\sigma\in L^1(0,T,L^\infty(\T))$ and $\theta\in L^\infty([0,T]\times \T)$. Then the system of equations \eqref{eq:souris}, \eqref{eq:chat} with initial conditions $\alpha_{0,\pm},\rho_{0,\pm}$ has at least one solution in $L^\infty([0,T]\times\T)^4$.
\end{prop}
\begin{proof}
As $u\in L^1(0,T,W^{1,\infty}(\T))$, for all $a,b\in L^1(0,T,L^\infty(\T))$ and $f_0\in L^\infty(\T)$, the system 
\begin{empheq}[left=\empheqlbrace]{alignat=1}
\partial_t f + u\partial_x f &= a f + b,\\
f(0,\cdot) &= f_0
\end{empheq}
has a unique solution in $L^\infty([0,T]\times\T)$. Indeed, introducing the characteristics associated to $u$, this Cauchy problem can be rewritten as a linear ODE. Choosing
\begin{equation}
a = -\frac{R_\pm\theta}{\mu_\pm}, \quad b = -\frac{\sigma}{\mu_\pm}, \quad f_0 = \rho_{0,\pm},
\end{equation}
We deduce the existence of some $\rho_\pm\in L^\infty([0,T]\times\T)$ such that
\begin{equation}
\partial_t\rho_\pm + u\partial_x \rho_\pm = -\frac{\sigma}{\mu_\pm}-\frac{R\rho_\pm\theta_\pm}{\mu_\pm},\quad \rho_\pm(0,\cdot) = \rho_{0,\pm}.
\end{equation}
Choosing then
\begin{equation}
a = \frac{\sigma}{\mu_\pm} + \frac{R_\pm\rho_\pm\theta}{\mu_\pm} -\partial_x u,\quad b = 0, \quad f_0=\alpha_{0,\pm}, 
\end{equation}
we get the existence of $\alpha_\pm$ solution of \eqref{eq:souris} with $\alpha_\pm(0,\cdot) = \alpha_{0,\pm}$. Finally, it is straightforward to verify that $(\alpha_\pm,\rho_\pm)$ verifies \eqref{eq:souris} and \eqref{eq:chat}.
\end{proof}
In our problem, the regularity of $u$, $\sigma$ and $\theta$ is sufficient to apply Proposition \ref{prop:existence} (see Proposition \ref{prop:twin}, Proposition \ref{prop:lala}). Let us define $\widetilde{\nu}:[0,T]\times\T\rightarrow {\cal M}(K)$ by
\begin{equation}
\text{for all }(t,x)\in[0,T]\times \T,\quad \widetilde{\nu}_{t,x} = \alpha_+(t,x)\delta_{1,\rho_+(t,x)} + \alpha_-(t,x)\delta_{0,\rho_-(t,x)},
\end{equation}
where $(\alpha_\pm,\rho_\pm)\in L^\infty([0,T]\times\T)^4$ is the solution of \eqref{eq:souris}--\eqref{eq:chat} with initial condition $(\alpha_{0,\pm},\rho_{0,\pm})$. Then $\widetilde{\nu}_{0,\cdot} = \nu_0$. Moreover, $\widetilde{\nu}\in C([0,T],{\cal M}(K,\T))$ because $D_t\alpha_\pm, D_t\rho_\pm \in L^\infty(0,T,L^1(\T))$ (see \cite{BrBuGJLa}, Lemma 5.1 for more details). 
\subsection{Characterization of the Young measure}
The aim of this subsection is to show that $\nu = \widetilde{\nu}$. Proposition \ref{prop:etonconclut} then concludes the proof of Theorem \ref{thm:main}. 
\begin{prop}
The Young measure $\nu$ satisfies the kinetic equation
\begin{equation}\label{eq:kineq}
\partial_t\nu + \partial_x(u\nu) = \left( \frac{\sigma}{\mu} + \frac{R\rho\theta}{\mu}\right)\nu + \partial_\rho\left(\left(\frac{\rho\sigma}{\mu} + \frac{R\rho^2\theta}{\mu}\right)\nu\right).
\end{equation}
\end{prop}
\begin{proof}
Let $\beta\in C^1(K)$. Then
\begin{align}
\partial_t\beta(c^\varepsilon,\rho^\varepsilon) + \partial_x(u^\varepsilon\beta(c^\varepsilon,\rho^\varepsilon)) &= (\beta(c^\varepsilon,\rho^\varepsilon) - \rho^\varepsilon\partial_\rho\beta(c^\varepsilon,\rho^\varepsilon))\partial_x u^\varepsilon
\\&= (\beta(c^\varepsilon,\rho^\varepsilon) - \rho^\varepsilon\partial_\rho\beta(c^\varepsilon,\rho^\varepsilon))\left(\frac{\sigma^\varepsilon}{\mu^\varepsilon} + \frac{R^\varepsilon\rho^\varepsilon}{\mu^\varepsilon}\theta^\varepsilon\right).
\end{align}
As a consequence, passing to the limit by using Proposition \ref{prop:twin} and Proposition \ref{prop:lala},
\begin{equation}
\partial_t \langle \nu, \beta\rangle + \partial_x(u\langle \nu,\beta\rangle) = \left\langle\nu, \frac{\beta - \rho\partial_\rho\beta}{\mu}\right\rangle\sigma + \left\langle \nu, \frac{R\rho(\beta - \rho\partial_\rho\beta)}{\mu}\right\rangle\theta.
\end{equation}
Note that, if $f,g\in C^1(K)$, then
\begin{equation}
\langle \nu, f g\rangle = \langle f\nu,g\rangle,
\end{equation}
\begin{equation}
\langle \nu, \partial_\rho f \rangle = -\langle \partial_\rho \nu, f\rangle, 
\end{equation}
thus \eqref{eq:kineq}.
\end{proof}
Moreover, it is straightforward to show that $\widetilde{\nu}$ is solution of \eqref{eq:kineq}. In order to prove that $\nu=\widetilde{\nu}$, it remains to prove the following uniqueness result for \eqref{eq:kineq}:
\begin{prop}\label{thm:kinuniq}
Let $\nu_0\in {\cal P}(K)$. Then the kinetic equation \eqref{eq:kineq} with initial condition $\nu_0$ has a unique solution in $C([0,T],{\cal M}(K\times\T))$. 
\end{prop}
\begin{proof}
The equation \eqref{eq:kineq} can be rewritten as 
\begin{equation}
\partial_t \nu + \Div(\bf{u} \nu) = \bf{g}\nu
\end{equation}
with $\Div = \Div_{x,c,\rho}$ and
\begin{align}
\text{for all }(t,x,c,\rho)\in [0,T]\times\T\times K,\quad &{\bf{u}}(t,x,c,\rho) = \left(u(t,x), 0, - \frac{\rho\sigma(t,x)}{\mu(c)} - \frac{R(c)\rho^2\theta(t,x)}{\mu(c)}\right),
\\&{\bf{g}}(t,x,c,\rho) = \frac{\sigma(t,x)}{\mu(c)} + \frac{R(c)\rho \theta(t,x)}{\mu(c)}.
\end{align}
We get
\begin{equation}
\nabla {\bf{u}} = \begin{pmatrix}
\partial_1 {\bf{u}}_1 & 0 \\
\partial_1 {\bf{u}}_2 & \nabla_2 {\bf{u}}_2
\end{pmatrix},\quad \nabla\bf{g} = \begin{pmatrix}
\partial_1 {\bf{g}}\\
\nabla_2 {\bf{g}}
\end{pmatrix},
\end{equation}
where
\begin{equation}
\partial_1 {\bf{u}}_1 = \partial_x u  \in L^1(0,T,L^\infty(\T\times K)),\quad  \partial_1{\bf{u}}_2 = \begin{pmatrix}
0 \\
-\dfrac{\rho\partial_x\sigma}{\mu} - \dfrac{R\rho^2\partial_x\theta}{\mu}
\end{pmatrix}\in L^1(0,T,L^1(\T\times K)),
\end{equation}
\begin{equation}
\nabla_2 {\bf{u}}_2 = \begin{pmatrix}
0 & 0\\
\dfrac{\rho\sigma(\mu_+ - \mu_-)}{\mu^2} - \dfrac{(R_+ - R_-)\rho^2\theta}{\mu} + \dfrac{R\rho^2\theta(\mu_+ - \mu_-)}{\mu^2}  & -\dfrac{\sigma}{\mu} - 2\dfrac{R\rho\theta}{\mu} 
\end{pmatrix}\in L^1(0,T,L^\infty(\T\times K)),
\end{equation}
\begin{equation}
\partial_1 {\bf{g}} = \frac{\partial_x\sigma}{\mu}\in L^1(0,T,L^1(\T\times K)),
\end{equation}
\begin{equation}
\quad \nabla_2 {\bf{g}} =
\begin{pmatrix}
\dfrac{-\sigma(\mu_+ - \mu_-)}{\mu^2} + \dfrac{(R_+ - R_-)\rho\theta}{\mu} - \dfrac{R\rho\theta(\mu_+ - \mu_-)}{\mu^2}\\
\dfrac{R\theta}{\mu}
\end{pmatrix}\in L^1(0,T,L^\infty(\T\times K)).
\end{equation}
Thus Lemma 5.2. in \cite{BrBuGJLa} can be applied, and gives Theorem \ref{thm:kinuniq}.
\end{proof}
\section*{Acknowledgements}  

This work was partially supported by the Agence Nationale pour la Recherche under France 2030 bearing the reference ANR-23-EXMA-004 (Complexflows).
The author would like to thank Didier Bresch and Frédéric Lagoutière for their attention and comments. 
\appendix
\section{Notations}\label{app:not}
\begin{notation}
For all $f:\T\rightarrow \R$ and $p\in [1,+\infty]$, we denote by $\Vert f\Vert_p$ the $L^p$ norm of $f$. If $f$ is also a function of $t\in [0,T]$, then $\Vert f\Vert_p$ is a function, $\Vert f\Vert_p : [0,T]\rightarrow [0,+\infty[$.
\end{notation}

\begin{notation}
For the rest of the analysis, let us introduce the total derivative $D_t$ formally defined for a function $f:\R_+\rightarrow \R$ by
\begin{equation*}
D_t f := \partial_t f + u\partial_x f.
\end{equation*}
Observe that from \eqref{eq:cont} we get
\begin{equation*}
\rho D_t f = \partial_t(\rho f) + \partial_x(\rho u f).
\end{equation*}
We then define a reciprocal at right operator $D_t^{-1}$ by
\begin{empheq}[left=\empheqlbrace]{align*}
D_t D_t^{-1} f &= f,\\
D_t^{-1} f(0,\cdot) &= 0.
\end{empheq}
Note that $D_t^{-1}$ is well defined supposing $\partial_x u\in L^1(0,T,L^\infty(\T))$, which corresponds well with the framework of the analysis. Then, defining
\begin{equation}\label{eq:20}
R_t f := f - D_t^{-1} D_t f,
\end{equation}
we get
\begin{empheq}[left=\empheqlbrace]{align*}
D_t (R_t f) &= 0,\\
(R_t f)(0,\cdot) & = f(0, \cdot),
\end{empheq}
then by some maximal principle (because $u\in L^1(0,T,W^{1,\infty}(\T))$)
\begin{equation}\label{eq:21}
\essinf_{x\in \T} f(0,x) \leq R_t f \leq \esssup_{x\in \T} f(0,x)\quad \text{almost everywhere}.
\end{equation}
In particular, if $f(0,\cdot) = 0$, then $R_t f = 0$ and $D_t^{-1} D_t f = f$. Moreover, $D_t^{-1}$ is a non-negative operator, and in particular, for all $(t,x)\in [0,+\infty[\times\T$,
\begin{equation}\label{eq:22}
\vert D_t^{-1} f\vert(t,x) \leq \int_0^t \Vert f\Vert_\infty.
\end{equation}
Remark moreover that
\begin{equation}\label{eq:23}
D_t^{-1} \int_0^1 f(t,x)dx = \int_0^t\int_0^1 f,
\end{equation}
and 
\begin{equation}\label{eq:rigolo}
D_t^{-1} (fc) = cD_t^{-1}f.
\end{equation}
\end{notation}

\begin{notation}
For a function $f:\T\rightarrow \R$, let us formally define $\partial_x^{-1}$ the right inverse operator of $\partial_x$, reciprocal at left on the space of zero mean functions, by
\begin{equation*}
\text{for all }x\in\T,\quad\partial_x^{-1} f (x) = \int_0^1dy\int_y^x f(z)dz.
\end{equation*}
In particular,
\begin{equation}\label{eq:24}
\vert\partial_x^{-1}f\vert\leq \int_0^1 \vert f\vert.
\end{equation}
Observe that $\partial_x^{-1}f$ is $1$-periodic if and only if $\int_0^1 f = 0$. Thus we get the formula
\begin{equation}\label{eq:25}
\partial_x^{-1}\partial_x f = f - \int_0^1 f.
\end{equation}
\end{notation}
\begin{notation}
If $f \in \{\mu,\cv,\gamma,R,\kappa\}$, we denote by

\[f_{\max} = \max(f(0), f(1)),\quad f_{\min} = \min(f(0),f(1)).\]
\end{notation}

\section{Justification of existence of solutions to  \texorpdfstring{\eqref{eq:cmeso}--\eqref{eq:thetameso}}{NS}}\label{app:justif}
\subsection{Local existence of strong solutions}

Let us recall the following local existence theorem by D. Serre \cite{Se10}:
\begin{theorem}[Serre]\label{thm:deserre}
Let $n\geq 1$ and $\mathcal{O}\subset \R^n$. Consider the hyperbolic-parabolic problem
\begin{equation}\label{eq:systlocal}
\left\{\begin{aligned}
\partial_t U + \partial_x F(U) &= \partial_x(B(U)\partial_x U),\\
U(0,\cdot) &= U_0, 
\end{aligned}
\right.
\end{equation}
where $F:\mathcal{O}\rightarrow \R^n$ and $B:\mathcal{O}\rightarrow M_n(\R)$ are smooth functions. Assume that
\begin{enumerate}
\item the associated hyperbolic problem
\begin{equation}
\partial_t U + \partial_x F(U) = 0
\end{equation}
admits an entropy-flux pair $(\eta,q)$ such that $\eta$ is strictly convex; 
\item the system \eqref{eq:systlocal} is \textit{strongly entropy dissipative} for $\eta$, in the sense that there exists a continuous and positive map $\omega : \mathcal{O}\rightarrow \R_+^*$ such that, for all $U\in \mathcal{O}$,
\begin{equation}
\text{for all }\xi\in\R^n,\quad \xi^\intercal B(U)^\intercal \nabla^2\eta(U) \xi \geq \omega \Vert B(U)\xi\Vert_2^2;
\end{equation}
\item for all $U\in \mathcal{O}$, the matrix $B(U)$ has the form
\begin{equation}
B(U) = \left(\begin{array}{cc}
0_{p\times n}\\
b(U)
\end{array}\right)
\end{equation}
where the rank of $b(U)\in M_{n-p,n}(\R)$  is always $n-p$;
\item the unique solution $Z(U)\in M_{n-p,n-p}(\R)$ of $b(U) = Z(U) J_p (\nabla^2\eta)$, where
\begin{equation}
J_p = \begin{pmatrix}
0_{p\times(n-p)}& I_{(n-p)\times(n-p)}
\end{pmatrix},
\end{equation}satisfies that for all $K\subset \mathcal{O}$ and every derivative $\partial$ of order $1$ or $2$, there exists $c_{K}>0$ such that, for all $U\in K$,
\begin{equation}
\text{for all }\xi\in\R^n,\quad\vert\partial Z(U) \xi\vert\leq c_{k,K}\vert Z(U)\xi\vert.
\end{equation}
\end{enumerate}
Then, given an initial data $u_0\in H^2(\T)$, there exists $T>0$ and a unique solution in the class
\begin{equation}
U\in C(0,T,H^2(\T)),\quad \partial_t U\in L^2(0,T,H^{1}(\T)).
\end{equation} 
\end{theorem}
\begin{remark}
This theorem is stated in \cite{Se10} in dimension $d\geq 1$, for a $H^s$ regularity where $s>1+d/2$. Moreover, D. Serre deals with the space $\R^d$ rather than $\T^d$, but his arguments can be adapted to the torus.
\end{remark}
Our goal is to apply Theorem \ref{thm:deserre} for the system \eqref{eq:cmeso}--\eqref{eq:energmeso}. Let us consider $\mathcal{O}_0\subset\subset [0,1]\times ]0,+\infty[\times\R\times ]0,+\infty[$ and $(c,\rho,u,\theta)\in \mathcal{O}_0$. Considering the change of variables
\begin{equation}
A_0 : \left|\begin{array}{ccc}
\mathcal{O}_0 &\rightarrow &\mathcal{O}\\
U_0 = \begin{pmatrix} c&\rho&u&\theta\end{pmatrix}^\intercal &\mapsto &U=\begin{pmatrix}\rho_c&\rho&m&j\end{pmatrix}^\intercal,
\end{array}\right.
\end{equation}
where
\begin{equation}
U = A_0(U_0) =\begin{pmatrix}\rho c&\rho&\rho u&\rho\cv(c)\theta + \rho u^2/2 \end{pmatrix}^\intercal,
\end{equation}
the system \eqref{eq:cmeso}--\eqref{eq:energmeso} can be rewritten as
\begin{equation}\label{eq:247}
\partial_t U + \partial_x F(U) = \partial_x(B(U)\partial_x U),
\end{equation}
with
\begin{equation}
F(U) = \begin{pmatrix}
\rho c u & \rho u & \rho u + R(c)\rho\theta & \rho u^3/2 + \rho \cv(c)\theta u + R(c)\rho \theta u
\end{pmatrix}^\intercal,
\end{equation}
\begin{equation}
B(U) = \begin{pmatrix}
0 & 0 & 0 & 0\\
0 & 0 & 0 & 0\\
0 & -\dfrac{\mu(c)}{\rho} & \dfrac{\mu(c)}{\rho} & 0\\
-\dfrac{(\cv_+-\cv_-)\kappa(c)\theta}{\cv(c)\rho} & -\dfrac{\mu(c) u^2}{\rho} - \dfrac{(2\cv_-\theta - u^2)\kappa(c)}{2\cv(c)\rho} & \dfrac{\mu(c) u}{\rho} - \dfrac{\kappa(c) u}{\cv(c)\rho} & \dfrac{\kappa(c)}{\cv(c)\rho}
\end{pmatrix}.
\end{equation}
\subsubsection{Construction of a strictly convex entropy}

A natural entropy associated with \eqref{eq:247} is
\begin{equation}
\eta(U) = \rho\cv(c)\ln\theta - R(c)\rho\ln\rho.
\end{equation}
However, although the physical entropy is classically strictly convex for the hyperbolic-parabolic Navier-Stokes system, this property no longer holds for the present augmented system due to the introduction of the color function $c$. Indeed, let us fix $\widetilde{c}\in [0,1]$ and introduce
\begin{equation}
L_{\widetilde{c}} = \begin{pmatrix}
\widetilde{c} & 0 & 0\\
1 & 0 & 0\\
0 & 1 & 0\\
0 & 0 & 1
\end{pmatrix},
\end{equation}
such that $(\eta\circ L_{\widetilde{c}})(\rho,m,j) = \eta(\rho\widetilde{c},\rho,m,j)$. Roughly speaking, we "freeze" the color function to get back to the classical Navier-Stokes system. It is well known that the entropy $\eta\circ L_{\widetilde{c}}$ associated to the reduced system is a strictly convex function, so there exists a continuous map $\omega_1 = \omega_1(U)>0$ such that
\begin{equation}
\text{for all }\xi\in\R^4,\quad \xi^\intercal L_{\rho_c/\rho}^\intercal (\nabla^2 \eta)(\rho_c,\rho,m,j) L_{\rho_c/\rho}\xi \geq \omega_1 \vert\xi\vert^2. 
\end{equation}
However, denoting
\begin{equation}
v = \begin{pmatrix}
1 & -\rho_c/\rho & 0 & 0
\end{pmatrix},
\end{equation}
we get $\im(L_{\rho_c/\rho})^\perp = \R v$ and
\begin{equation}
\begin{split}
v^\intercal (\nabla^2\eta)v&= \frac{1}{\cv(c)\rho}\left((1+c^2)(\cv_+-\cv_-) + \frac{u^2\rho c^2}{2\theta}\right)^2 + \frac{\cv(c)\gamma(c) c^2}{\rho}\\&- \frac{{2 {\gamma(c)}} {\left(1 + c^{2}\right)} {(\cv_+ - \cv_-)} {c}}{ {\rho}} - \frac{2 {\left(1 + {c}^{2}\right)} {\cv(c)} {(\gamma_+-\gamma_-)} {c}}{{\rho}}  \\&+ \frac{2 \, {\left(1 + c^2\right)}^{2} {(\cv_+-\cv_-)(\gamma_+-\gamma_-)} \log\left({\rho}\right)}{{\rho}},
\end{split}
\end{equation}
which is not positive in general (for example, taking $u=0$, $\cv_+ = \cv_-$ and $c$ small enough, we can even obtain negative values). 

To recover strict convexity, we augment the entropy with a quadratic stabilizer related to the conservation of the color function mass. Defining
\begin{equation}
\pi(y,\rho,m,j) = \frac{\rho_c^2}{2\rho},
\end{equation}
we get
\begin{equation}
\partial_t \pi + \partial_t (\pi u) = 0,
\end{equation}
and
\begin{equation}
\nabla^2\pi(U) = \frac{1}{\rho}\begin{pmatrix}
1 & -c & 0 & 0\\
-c & c^2 & 0 & 0\\
0 & 0 & 0 & 0\\
0 & 0 & 0 & 0
\end{pmatrix}.
\end{equation}
Note that $\ker(\nabla^2\pi) = \im(L_{\rho_c/\rho})$. Moreover,
\begin{equation}
v^\intercal (\nabla^2\pi) v = \frac{1 + c^4}{\rho}>0. 
\end{equation}

Let $\beta>0$. Let $u\in\R^4$ and $\xi = L_{\rho_c/\rho} u + v$. Then, by Cauchy-Schwarz's inequality and Young's inequality, 
\begin{align}
\xi^\intercal (\nabla^2 \eta + \beta\nabla^2 \pi)\xi &= \xi^\intercal (\nabla^2\eta)\xi + \beta\xi^\intercal (\nabla^2\pi)\xi
\\&=u^\intercal L_{\rho_c/\rho}^\intercal (\nabla^2\eta) L_{\rho_c/\rho} u + 2 v^\intercal (\nabla^2\eta)L_{\rho_c/\rho} u+ v^\intercal (\nabla^2\eta) v+ \beta v^\intercal (\nabla^2\pi)v
\\&\geq \omega_1 \vert L_{\rho_c/\rho} u\vert^2 - 2\Vert \nabla^2\eta\Vert\vert v\vert \vert L_{\rho_c/\rho}u\vert +  v^\intercal (\nabla^2\eta) v + \beta v^\intercal (\nabla^2\pi) v
\\&\geq \frac{\omega_1}{2}\vert L_{\rho_c/\rho}u\vert^2 -\frac{2\Vert\nabla^2\eta\Vert^2}{\omega_1}\vert v\vert^2 - \Vert\nabla^2\eta\Vert\vert v\vert^2 + \beta v^\intercal (\nabla^2\pi)v.
\end{align}
Thus, choosing
\begin{equation}
\beta > \sup_{U\in \mathcal{O}}\rho\vert v\vert^2 \left(\frac{2\Vert\nabla^2\eta\Vert^2}{\omega_1} + \Vert\nabla^2\eta\Vert\right),
\end{equation}
the function $\eta + \beta \pi$ is strictly convex.

\subsubsection{Checks on other assumptions of the theorem}
We still need to check the assumptions $2$, $3$ and $4$ of Theorem \ref{thm:deserre}. 

First, remark that
\begin{equation}
B^\intercal \nabla^2\pi = 0,
\end{equation}
so to show that assumption $2$ is verified, we just have to prove that there exists some $\omega:\mathcal{O}\rightarrow \R_+^*$ such that
\begin{equation}\label{eq:reduced}
\text{for all }\xi\in\R^4,\quad \xi^\intercal B^\intercal \nabla^2\eta\xi \geq \omega \vert B\xi\vert^2.
\end{equation}

Direct computations show that, for $\xi\in\R^4$,
\begin{equation}
\vert B \xi\vert^2 = 4(u^2+1)c_v^2\mu^2 \varphi_1(\xi)^2 + 4\frac{\varphi_1(\xi)\varphi_2(\xi)\cv\kappa\mu u}{\rho} + \frac{\kappa^2}{\rho^2}\varphi_2(\xi)^2
\end{equation}
\begin{equation}\label{eq:bientotvisio}
\xi^\intercal B^\intercal \nabla^2\eta \xi = 4 c_v^2\mu\rho\theta\varphi_1(\xi)^2 + \frac{\kappa}{\rho}\varphi_2(\xi)^2,
\end{equation}
with 
\begin{equation}
\varphi_1(\xi) = u \xi_2 - \xi_3,
\end{equation}
\begin{equation}
\varphi_2(\xi) = 2(\cv_+ - \cv_-)\rho\theta \xi_1 + (1-2 c \cv_--\rho u^2)\xi_2 + 2\rho u\xi_3 - 2\rho\xi_4.
\end{equation}
Thus there exists some continuous $\omega$ such that \eqref{eq:reduced} holds, as the coefficients behind $\varphi^2$ and $\varphi_2^2$ in \eqref{eq:bientotvisio} are positive.

Secondly, assumption $3$ is clearly verified with $n=4$ and $p=2$, since $\mu(c)/\rho\neq 0$ and $\kappa(c)/(\cv(c)\rho)\neq 0$.

To prove $4$, note that $Z(U)$ is solution to
\begin{equation}\label{eq:mar}
b(U) = Z(U)J_p\nabla^2(\eta+\beta\pi) = Z(U)J_p\nabla^2\eta. 
\end{equation}
As the rank of $J_p\nabla^2\eta$ and $b(U)$ are $n-p$, $Z(U)$ is invertible. Multiplying \eqref{eq:mar} by the right inverse of $J_p\nabla^2\eta$, $Z(U)$ can be expressed as a smooth function of $U$, the same holds for $Z(U)^{-1}$, and for $\partial Z(U)$. Thus, for all $K\subset {\mathcal O}$, there exists $C_K>0$ such that
\begin{equation}
\text{for all }U\in K,\quad \Vert Z(U)^{-1}\Vert\leq C_K,\quad \Vert\partial Z(U)\Vert\leq C_K.
\end{equation}
The proof of $4$ is then straightforward, since
\begin{align}
\text{for all }\xi\in\R^{n-p},\quad \vert \partial Z(U)\xi\vert &\leq \Vert \partial Z(U)\Vert \vert \xi\vert \\&\leq \Vert\partial Z(U)\Vert\Vert Z(U)^{-1}\Vert \vert Z(U)\xi\vert.
\end{align}

\subsection{Global existence of strong solutions}
\
The bounds of Theorem \ref{thm:bounds} are valid for the solution of Theorem \ref{thm:deserre}. We can then improve these bounds for smooth initial data:
\begin{prop}
Let $(c,\rho,u,\theta)\in C(0,T_0,H^2(\T))$ be a solution of \eqref{eq:cmeso}--\eqref{eq:thetameso}. Then there exists $C>0$ depending only on $\Vert (c(0,\cdot),\rho(0,\cdot),u(0,\cdot),\theta(0,\cdot))\Vert_{H^2}$ such that
\begin{equation}
\Vert (c,\rho,u,\theta)\Vert_{L^\infty_{T_0}(H^2)}\leq C. 
\end{equation}
\end{prop}
We omit the proof of this proposition because it is tedious but classical. The main ingredient is the use of the second energy of Hoff (see \cite{Li2}). By a bootstrap argument, we then obtain the
\begin{theorem}
Suppose that $(c_0,\rho_0,u_0,\theta_0)\in H^2(\T)$ and that $(c_0,\rho_0,u_0,\theta_0)\in {\cal O}$. Then, for all $T>0$, \eqref{eq:cmeso}--\eqref{eq:thetameso} with initial condition $(c_0,\rho_0,u_0,\theta_0)$ has a unique solution $(c,\rho,u,\theta)\in C(0,T,H^2(\T))$.
\end{theorem}  
\subsection{Global existence of "à la Hoff" solutions}
\
Finally, by regularizing the initial data by convolution and using the uniform bounds from Theorem \ref{thm:bounds}, we establish the global existence of "à la Hoff" solutions to \eqref{eq:cmeso}--\eqref{eq:thetameso}. Thanks to the strong compactness of $u, \theta$ and  $\sigma$, passing to the limit in the nonlinear terms of the Navier-Stokes system is straightforward. We refer to \cite{BrBuGJLa} for the proof of the strong compactness of the color function $c$.

\printbibliography

\end{document}